\documentclass[10pt]{amsart}
\usepackage{amssymb, amsmath, url,color,  amsthm, amssymb,graphicx,multirow, appendix}
\usepackage[margin=1in]{geometry}
\usepackage{bigints}
\usepackage{BOONDOX-cal}
\usepackage{lscape,amssymb, amsmath,verbatim,graphicx}
\usepackage{epsfig,float}
\usepackage{graphicx}
\usepackage{epsfig}
\usepackage{threeparttable}
\usepackage{ragged2e}
\usepackage{multirow}
\usepackage{adjustbox}
\usepackage{booktabs}

\usepackage{mathtools}
\mathtoolsset{showonlyrefs,showmanualtags}

\def\beq{\begin{equation}}
\def\eeq{\end{equation}}
\usepackage[round]{natbib}

\let \cite = \citep
\newtheorem{assumption}{Assumption}[section]
\newtheorem{example}{Example}[section]
\usepackage{mathrsfs}
\DeclareMathOperator*{\argmax}{argmax}
\usepackage{array}

\usepackage{multicol}
 \usepackage{subfig}
  \usepackage[edges]{forest}
 \usepackage{xr}
 \usepackage{bbm}

\newtheorem{theorem}{Theorem}[section]
\newtheorem{lemma}{Lemma}[section]

\newtheorem{Remark}{Remark}[section]
\newtheorem{Proposition}{Proposition}[section]

\newcommand{\lf}{\lfloor}
\newcommand{\rf}{\rfloor}

\DeclareMathOperator{\Var}{Var}
\DeclareMathOperator{\var}{var}

\renewcommand{\P}{P}
\usepackage{BOONDOX-cal}

\newcommand{\cm}{\mathcal m}

\newcommand{\ca}{\mathcal a}
\newcommand{\cs}{\mathcal s}
\newcommand{\cU}{\mathcal U}

\newcommand{\cmm}{\mathcal m}
\newcommand{\ct}{\mathcal t}
\newcommand{\ccr}{\mathcal r}

\newcommand\T{\top}

\newcommand{\bX}{{\bf X}}

\newcommand{\bZ}{{\bf Z}}

\newcommand{\cD}{{\mathcal  D}}

\newcommand{\bx}{{\bf x}}

\newcommand{\E}{E} %

\numberwithin{equation}{section}
\numberwithin{Remark}{section}
\numberwithin{Assumption}{section}
\numberwithin{lemma}{section}
\allowdisplaybreaks

\usepackage{chngcntr}

\makeatletter
\@namedef{subjclassname@2020}{%
  \textup{2020} Mathematics Subject Classification}
\makeatother
\usepackage[foot]{amsaddr}

\usepackage{enumerate}

\begin{document}

\newcommand\plim{\mathop{p\mkern2mu\mathrm{\mathchar"702D lim}}}

\title[Change point detection in high dimensional data]{Change point detection in high dimensional data with U-statistics   }

\author{B.\ Cooper Boniece$^{1,*}$}
\email{cooper.boniece@drexel.edu}
\address{$^{1}$Department of Mathematics, Drexel University, Philadelphia, PA
19104 USA }
\author{Lajos Horv\'ath$^2$}
\email{horvath@math.utah.edu}
\address{$^2$Department of Mathematics, University of Utah, Salt Lake City, UT 84112--0090 USA }
\author{Peter M.\ Jacobs$^3$}
\address{$^3$School of Computing, University of Utah, Salt Lake City, UT 84112--0090 USA }
\email{u1266560@utah.edu}

\subjclass[2020]{Primary 62G20, Secondary   60F17, 62G10, 62H15, 62P25}
\keywords{large dimensional vectors, U-statistics, weak convergence, change point, Twitter data}

\begin{abstract}

We consider the problem of detecting distributional changes in a sequence of high dimensional data.  Our approach combines two separate statistics stemming from $L_p$ norms whose behavior is similar under $H_0$ but potentially different under $H_A$, leading to a testing procedure that that is flexible against a variety of alternatives.   We establish the asymptotic distribution of our proposed test statistics separately in cases of weakly dependent and strongly dependent coordinates as $\min\{N,d\}\to\infty$, where $N$ denotes sample size and $d$ is the dimension, and establish consistency of testing and estimation procedures in high dimensions under one-change alternative settings.  Computational studies in single and multiple change point scenarios demonstrate our method can outperform other nonparametric approaches in the literature for certain alternatives in high dimensions. We illustrate our approach through an application to Twitter data concerning the mentions of U.S. Governors. 
\end{abstract}

\maketitle

\section{Introduction}

Detecting changes in a given sequence of data is a problem of critical importance in a variety of disparate fields and has been studied extensively in the statistics and econometrics literature for the past 40+ years.  Recent applications include finding changes in terrorism-related online content  \citep{theo},  intrusion detection for cloud computing security \citep{aldribi:traore:etal:2020}, and monitoring emergency floods through the use of social media data \citep{shoyama:cui:etal:2021}, among many others. The general problem of change point detection may be considered from a variety of viewpoints; for instance, it may be considered in either ``online'' (sequential) and ``offline'' (retrospective) settings, under various types of distributional assumptions, or under specific assumptions on the type of change points themselves. See,  for example, \citet{horvath:rice:2014} for a survey on some traditional approaches and some of their extensions.

The importance of traditional univariate and multivariate contexts notwithstanding, it is increasingly common in contemporary applications to encounter \textit{high-dimensional} data whose dimension $d$ may be comparable or even substantially larger than the number of observations $N$. Popular examples include applications in genomics  \citep{amaratunga:cabrera:2018},  or in the analysis of social media data \citep{gole:tidke:2015}, where $d$ can be up to several orders of magnitude larger than $N$.  However many classical inferential methods provide statistical guarantees only in ``fixed-$d$" large-sample asymptotic settings that implicitly require the sample size $N$ to overwhelm the dimension $d$,  rendering several traditional approaches to change-point detection unsuitable for modern applications in which $N$ and $d$ are both large.  Accordingly, there has been a surge of research activity in recent  years concerning methodology and theory for change-point detection in the asymptotic setting most relevant for applications to high dimensional data, i.e., where both $N,d\to\infty$ in some fashion;  see \citet{liu:zhang:etal:2022} for a survey regarding new developments.  Commonly, asymptotic results in this context require technical restrictions on the size of $d$ relative to $N$, ranging from more stringent conditions such as $d$ having logarithmic-type or polynomial growth in $N$ \citep{jirak:2012}, to milder conditions that permit $d$ to have possibly exponential growth in $N$ \citep{liu:zhou:etal:2020}.  However, for maximal flexibility in practice, it is desirable to have methods that require as little restriction as possible on the rate at which $d$ grows relative to $N$.

In this work, we are concerned with change-point detection problem in the ``offline'' setting in which a given sequence of historical data is analyzed for the presence of changes.
Specifically, we are concerned with the following: let $\bX_1, \bX_2, \ldots, \bX_N$ be random vectors in $\mathbb R^d$ with distribution functions $F_1(\bx), F_2(\bx), \ldots, F_N(\bx)$. We aim to test the null hypothesis %
$$
H_0:\;\;F_1(\bx) = F_2(\bx) = \ldots =F_N(\bx)\quad \mbox{for all}\;\bx\in \mathbb R^d
$$\vspace{-0.5ex}
against the  alternative
\begin{align*}%
H_A:\;\;\;&\mbox{there are}\;1=k_0<k_1<k_2<\ldots<k_R<k_{R+1}=N\;\mbox{and}\;\bx_1, \bx_2, \ldots,\bx_R\;~\mbox{such that}\; \\
&F_{k_{i-1}+1}(\bx)=F_{k_{i-1}+2}(\bx)=\ldots=
F_{k_{i}}(\bx), \quad 1\leq i \leq R+1\;\mbox{for all}\;\bx\in \mathbb R^d  \\
&\;\mbox{and}~ F_{k_{i}}(\bx_i)\neq F_{k_{i}+1}( \bx_{i}), \quad 1\leq i \leq R
\end{align*}%
where $R\geq 1$ is unknown.  The central contribution of this work is a flexible method for testing the hypotheses above whose asymptotic properties are supported theoretically in high dimensions;  namely, when $d$ is potentially large relative to $N$.  To accommodate this mathematically, we provide asymptotic statements in the asymptotic regime $\min\{N,d\}\to\infty$.     The novelty of our approach lies in combining two separate statistics stemming from $L_p$ norms  whose behavior is similar under $H_0$ but potentially different under $H_A$, each individual statistic measuring related but different aspects of the data, leading to a procedure that that is flexible for testing against a variety of alternatives.    To the best of our knowledge, it is the only method in this setting that retains relatively standard asymptotic behavior, thereby admitting critical values that are readily obtained. Ultimately, this leads to a straightforward asymptotic test and estimation procedure that is easy to implement and requires little restriction on the size of $d$ relative to $N$ for use in practice.

The problem of testing the hypotheses above has been been studied by several authors in recent years in both multivariate and high-dimensional settings;  e.g., \citet{lung-yut-fong:levy-leduc:etal:2011,chen:zhang:2015,arlot:celisse:etal,chu:chen:2019}, among others.  Concerning some approaches related to U-statistics, \citet{matteson:james:2014} proposed the use of empirical divergence measures based on the energy distance \cite{szr} assuming that $d$ is fixed.  Though their method has gained some popularity in applications and can perform quite well in certain settings, it lacks theoretical support in high dimensions and has some noteworthy drawbacks.  It has recently been shown by \citet{chakraborty:zhang:2021} and \citet{zhu:zhang:etal:2020} that such divergence measures are potentially unsuitable for detecting certain types of changes in data when the dimension $d$ is large, since they capture primarily only second-order structure in the data in certain settings and can be insensitive to changes beyond first and second moments (see also recent work of \citet{chakraborty:zhang:2021a} that addresses some of these issues).
	
	We also note test statistics employed in  \citet{matteson:james:2014} are related to so-named \textit{degenerate} U-statistics and therefore have a limit distribution that is nonstandard; an explicit expression for the limit (with fixed $d$) was first given by \citet{biau:bleakley:etal:2016} based on an infinite series representation that depends on a sequence of eigenvalues $\lambda_1,\lambda_2,\ldots$ that must be estimated from data for its practical implementation.  \citet{biau:bleakley:etal:2016} point out that resampling methods for such approaches can be computationally burdensome over large samples due to quadratic (in $N$) computational cost required by their method, illustrating asymptotic tests for U-statistic-based approaches can be especially advantageous over larger sample sizes.

 Recently, \citet{liu:zhou:etal:2020} proposed a flexible framework for detecting change points based on $q$-dimensional U-statistics in the setting where $q,d,N\to\infty$.  Though their method is quite flexible, the authors do not obtain the limit distribution of their test statistics and rely on high-dimensional bootstrapping methods to obtain critical values.  Our detection method is based on somewhat simpler one-dimensional U-statistics whose distribution may be explicitly obtained,  bypassing the need for bootstrapping, permutation tests, or similar methods.%

This paper is organized as follows.  Section \ref{s:main} contains framework, notation, and a brief discussion regarding the principle behind our approach.  Section \ref{s:theoreticalresults} contains our supporting theory in high dimensions and discussions of practical implementation of the test statistics.   Section \ref{sec-sim} contains simulation studies, and Section \ref{sec-appl} contains an illustration of our method in an application based on Twitter data concerning mentions of the U.S. governors. Appendix \ref{app-ex} contains some important examples, and all proofs are given in Appendices \ref{sec-pr1} through  \ref{app-mom}.

\section{Framework}\label{s:main}
In what follows, $\|\cdot\|_p$ denotes the vector $L_p$ norm in $ \mathbb R^d$ for some arbitrary fixed $1\leq p <\infty$, and   $\lf x\rf$ denotes the integer part of $x\in \mathbb R$.  For $x<0$ we denote $x^\beta=-(-x)^\beta$ for any $\beta\in \mathbb R$, and for any sequences $a_N,b_N> 0$, we write $a_N\gg b_N$ if $a_N b^{-1}_N \to \infty$ as $N\to\infty$.
\noindent We work with the (idealized) assumption
\begin{assumption}\label{as2} \;$\bX_1, \bX_2, \ldots, \bX_N$ are independent random vectors.
\end{assumption}
\noindent Assumption \ref{as2} is in force throughout the paper.

Our proposed method is based on weighted functionals of two processes constructed from U-statistics.   We first define some intermediate quantities. 
 For $2\leq k <N-2$, we set
 \begin{equation}
  \begin{gathered}
U_{N,d,1}(k)=\frac{1}{\displaystyle
\begin{pmatrix}
k\\
2
\end{pmatrix}
}
\sum_{1\leq i<j\leq k}\|\bX_i-\bX_j\|_p,\quad \quad
U_{N,d,2}(k)=\frac{1}{\displaystyle
\begin{pmatrix}
N-k\\
2
\end{pmatrix}
}
\sum_{k+1\leq i<j\leq N}\|\bX_i-\bX_j\|_p,\\
U_{N,d,3}(k)=\frac{1}{k(N-k)}
\sum_{i=1}^k\sum_{j=k+1}^N\|\bX_i-\bX_j\|_p, \quad\quad U_{N,d,4}=
\frac{1}{N^2}
\sum_{i=1}^N\sum_{j=1}^N\|\bX_i-\bX_j\|_p.
\label{e:Ustats}
\end{gathered}
\end{equation}
We now proceed to define two processes  $V_{N,d}=\{V_{N,d}(t), 0\leq t \leq 1\}$ and $Z_{N,d} =\{Z_{N,d}(t),0\leq t \leq 1\}$,  each meant to capture different aspects of the data.  They are constructed based on the differences of the statistics $U_{N,d,\ell}(k)$ (suitably normalized to account for the differing sample sizes among each statistic $U_{N,d,\ell}(k)$ as $k$ runs through the set $2\leq k \leq N-2$.)  For each $t\in[2/N,1-  2/N]$, let\\
\begin{align*}
V_{N,d}(t)=t(1-t)d^{-1/p} \big[U_{N,d,1}(\lf Nt\rf)-U_{N,d,2}(\lf Nt\rf)\big],%
\end{align*}
and for any $0\leq\beta<1$, set
$$
Z_{N,d}(t)=2 t(1-t)\big(|1-2t| +\tfrac{1}{\sqrt N}\big)^{-\beta} d^{-1/p}\left[U_{N,d,3}(\lf Nt\rf)-U_{N,d,4}\right]%
$$ \ \\
  and we set  $V_{N,d}(t)=Z_{N,d}(t)=0$ if $t\not\in [2/N, 1-2/N]$.  Our approach is based on the following test statistic:\\
\begin{equation}\label{e:main_test_statistic}
T_{N,d}= \sup_{0<t<1}\frac{\max \{|V_{N,d}(t)|,|Z_{N,d}(t)|\} }{w(t)}%
\end{equation}\ \\
 In \eqref{e:main_test_statistic},  $w(t)$ is a \textit{weight function} defined on $[0,1]$  satisfying the following properties:
\begin{enumerate}[$(i)$]
\item  $\inf_{\delta\leq t\leq 1-\delta}w(t)>0$ for all $0<\delta<1/2$%
\item  $w(t)$ is nondecreasing in a neighborhood of 0 and nonincreasing in a neighborhood of 1.%
\end{enumerate}

\begin{figure}
\begin{center}

\begin{minipage}{0.4\linewidth}
\centering
\includegraphics[width=0.65\linewidth]{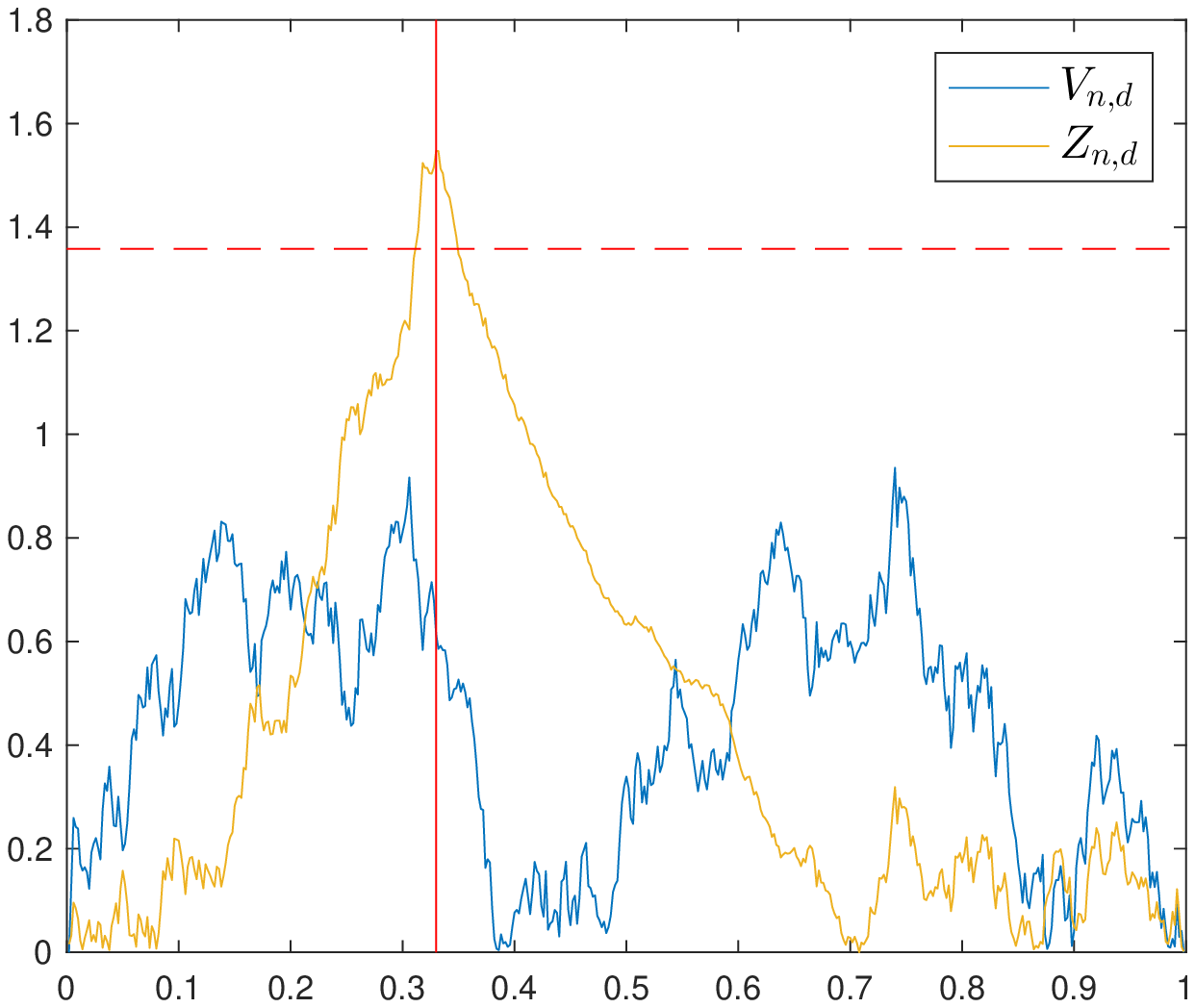}
\end{minipage}
\begin{minipage}{0.4\linewidth}
\centering
\includegraphics[width=0.65\linewidth]{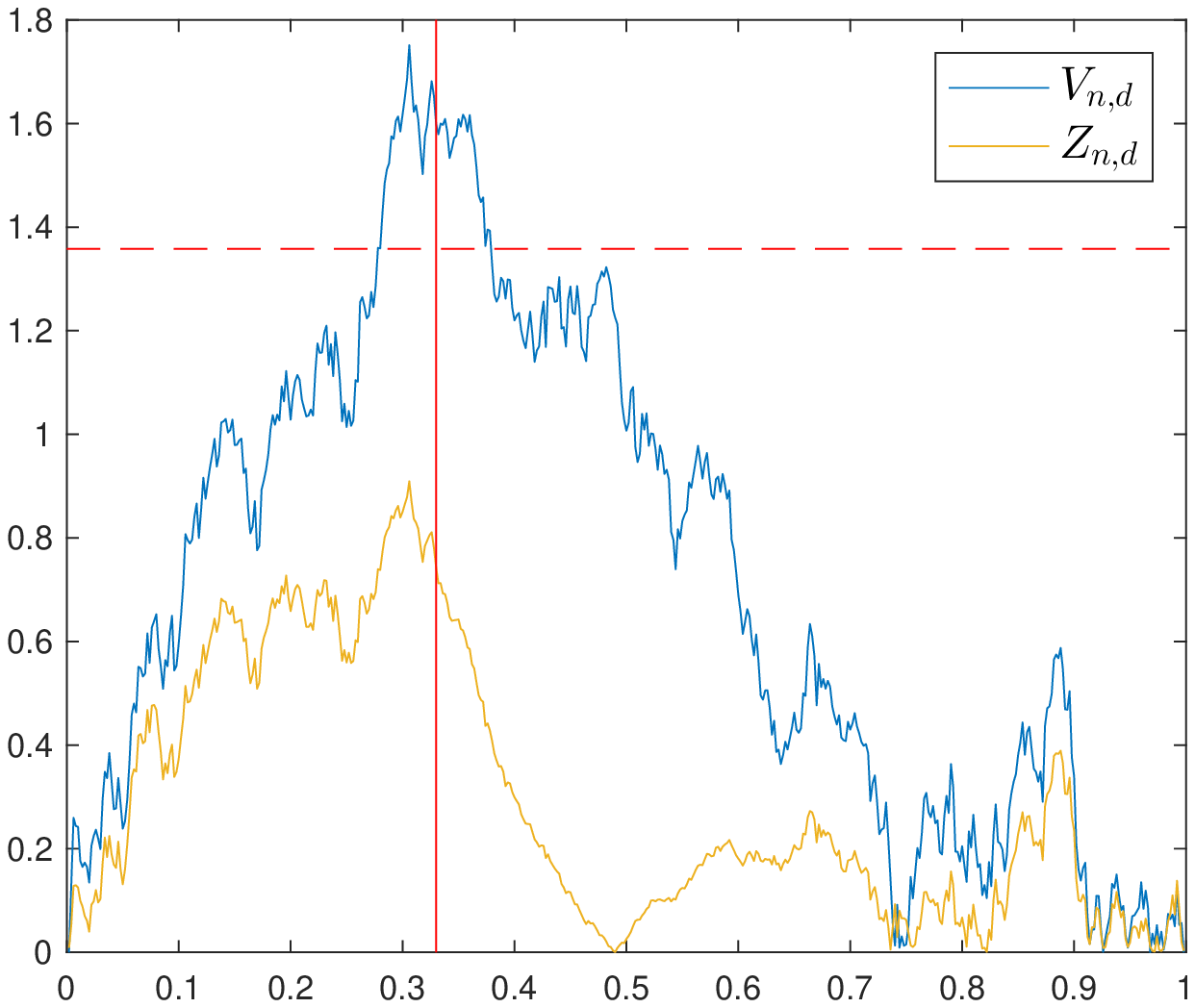}
\end{minipage}
\end{center}
\caption{\label{fig:idea}  \small Paths of $|Z_{n,d}(t)|$ and $|V_{N,d}(t)|$ with $w(t)\equiv 1$, $p=1$ based on a single change at $k_1=\lfloor N/3\rfloor $ (vertical line) for independent Gaussian observations, $N=d=500$; horizontal dashed line is associated critical value at size $0.05$. In the left figure there is only shift in the mean; on the right, only a change in scale. Jointly maximizing $Z_{N,d}$ and $V_{N,d}$ leads to increased power against a broader array of alternatives in high dimensions.} 
\end{figure}

Though tests based on either $Z_{N,d}$ or $V_{N,d}$ alone are suitable themselves, use of both $Z_{N,d}$ and $V_{N,d}$ together allows for increased power against a broader array of alternatives;  see Figure \ref{fig:idea}. Each is suitable for detecting changes of various types, but $V_{N,d}$ is especially sensitive to scale changes, whereas $Z_{N,d}$ is more sensitive to location changes, and both take into account information concerning $p$-th moments in different capacities.  Combining them by jointly maximizing them leads to a testing approach and estimation procedure that are sensitive to a larger variety of distributional changes in high-dimensional contexts.  (Note $T_{N,d}$ can be interpreted a test statistic based on a union-intersection test \citep{roy:1953} formed from two tests: one based on $V_{N,d}$ and one based on $T_{N,d}$.)  In contrast, both $Z_{N,d}$ and $V_{N,d}$ behave somewhat similarly under $H_0$, and owing to this, the statistic $T_{N,d}$ in fact enjoys standard limit behavior in high dimensions under $H_0$, leading to a straightforward asymptotic test. %

The use of the parameter $\beta$ and weight functions $w(t)$ also allows  for increased power against certain alternatives regarding type or location of change point(s). Very roughly speaking, the parameter $\beta$ can be viewed a proxy for the relative importance of alternatives measured by $Z_{N,d}(t)$ compared to those measured by $V_{N,d}(t)$, and the weight function $w(t)$ can be chosen to boost power when change points may be near the boundary of an interval. We further discuss the effect of selection of $w(t)$ and the auxiliary parameter $\beta$  in Section \ref{sec-sim}.

\section{Theoretical results} \label{s:theoreticalresults}
In the statements that follow, we make use of the integral functional
\begin{equation}\label{e:def_I(w,c)}
I(w,c)=\int_0^1 \frac{1}{t(1-t)}\exp\left( -\frac{cw^2(t)}{t(1-t)}   \right)dt
\end{equation}
to determine necessary and sufficient conditions to obtain a finite limit for the weighted supremum functionals of $V_{N,d}(t)$ and $Z_{N,d}(t)$. 

 \subsection{Size}\label{s:size} We  separately consider two distinct settings concerning possible types of dependence in the coordinates of the $\mathbf X_i$. %
 We defer discussion of implementation of the tests in practice to Section \ref{s:implement}.   
\subsubsection{Weak dependence}\label{s:weakdep}

We first turn to the case of weakly dependent coordinates.  For the statements below, we define the functions

\begin{equation}\label{def:gj}
g_j(x_j)=E|X_{1,j}-x_j|^p,\;\;1\leq j \leq d,
\end{equation}
where $x_1,\ldots,x_d\in \mathbb R$.  
\begin{assumption}\label{as4}\;For some $\alpha>\max\{p,10-4/p, 14-8/p\}$, and some constant $C>0$ independent of $d$ and $N$, we have $\max_{1\leq j \leq d}E|X_{1,j}|^\alpha\leq C$, 
\beq\label{as4-1}
E\bigg|\sum_{j=1}^d\big[\left| X_{1,j}-X_{2,j}  \right|^p-E\left| X_{1,j}-X_{2,j}  \right|^p\big]\bigg|^\alpha\leq Cd^{\alpha/2},
\eeq
and
\beq\label{as4-2}
E\bigg|\sum_{j=1}^d\big[ g_j(X_{1,j})  -E  g_j(X_{1,j}) \big]\bigg|^\alpha\leq Cd^{\alpha/2}.
\eeq
Moreover, the limits
\beq\label{as4-3}
\ca(p)=\lim_{d\to\infty}\frac{1}{d}\sum_{j=1}^dEg_j(X_{1,j}),\quad \tau^2=\lim_{d\to \infty}\frac{1}{d}E\bigg( \sum_{j=1}^d [ g_j(X_{1,j})-Eg_j(X_{1,j})]  \bigg)^2
\eeq
exist.
\end{assumption}
 Expressions \eqref{as4-1} and  \eqref{as4-2}  roughly state that the coordinates of $\mathbf X_{1}$ and of $\bX_1-\bX_2$ behave like weakly dependent random variables. In particular,  they are satisfied for AR(1)-type coordinates (c.f. Example \ref{exe2}). In Appendix \ref{app-ex}, we give explicit examples of settings in which Assumptions \ref{as4} and \ref{as5} are satisfied.  Note when $g_1(X_{1,1}), g_2(X_{1,2}),\ldots$ is a wide-sense stationary sequence, $\ca(p)$ is its mean and $\tau^2$ is its corresponding long-run variance. Furthermore, note the distribution of $\bX_1$ depends on $d$ and is also allowed to depend on $N$, so our assumptions  allow for triangular arrays of random vectors.

Next we define the limiting variance of $V_{N,d}(t)$ and $Z_{N,d}(t)$ for the case of weakly dependent coordinates.  Let
\beq\label{sidef}
\sigma^2=\left(\frac{2}{p}\ca^{-1+1/p}(p)\right)^2\tau^2.
\eeq
Note $\sigma^2$ is well-defined under Assumption \ref{as4}.
Since we will normalize with $\sigma^2$, it is natural to require
\begin{assumption}\label{as5}\;$\sigma^2>0$.
\end{assumption}
Assumption \ref{as5} amounts to a high-dimensional non-degeneracy condition. Note that similar assumptions are typically not met in fixed dimensions when the underlying U-statistics are of degenerate type.

Theorems \ref{th1}  and \ref{th2}, stated next, are our main results in the context of Assumption \ref{as4}, which provide the asymptotic distribution under of our test statistics under $H_0$ after appropriate rescaling. This lays the foundation for our asymptotic testing procedure. %

\begin{theorem}\label{th1} Assume that $H_0$ holds, and Assumptions \ref{as2}, \ref{as4} and \ref{as5} are satisfied. If $I(w,c)<\infty$ in \eqref{e:def_I(w,c)} is finite for some $c>0$, then as $\min\{N,d\}\to\infty,$
\begin{align}\label{th1-1}
\sigma^{-1} (Nd)^{1/2} T_{N,d}\stackrel{\cD}{\to}\sup_{0<t<1}\frac{|B(t)|}{w(t)},
\end{align}
where $\{B(t),0\leq t \leq 1\}$ is a Brownian bridge.%

\end{theorem}

Theorem \ref{th1} allows for a variety of possible weight functions $w(t)$;  since the limit in \eqref{th1-1} is finite with probability 1 if and only if $I(w,c)<\infty$ with some $c>0$, the conditions on the weight function for Theorem \ref{th1} are sharp (see \citet{csh-1}), as no other classes of weight functions can lead to such a limit.

However, it is sometimes desirable to self-normalize maximally selected statistics, i.e., to use a weight function that is proportional to the standard  deviation of the limit, at least in a neighborhood of 0 and 1. In this case, the corresponding weight function is $w(t)=(t(1-t))^{1/2}$, but Theorem \ref{th1} cannot be applied since $I((t(1-t))^{1/2}, c)=\infty$ for all $c>0$.  Thus, Theorem \ref{th2}, stated next, can be regarded as a counterpart to Theorem \ref{th1} for the choice of weight function $w(t)= (t(1-t))^{1/2}$.  It is a nonstandard Darling--Erd\H{o}s--type result,  showing Gumbel-type limit behavior emerges under the choice $w(t)=(t(1-t))^{1/2}$.

Before giving its statement, we define some auxiliary quantities based on the projections of U-statistics (e.g., \citet{lee:1990}). 
 For $\bx=(x_1, x_2, \ldots, x_d)$, let
\begin{equation}\label{e:H(x)}
H(\bx)=E\left(\frac{1}{d}\sum_{\ell=1}^d|X_{1,\ell}-x_\ell|^p\right)^{1/p},\quad \theta=E [H(\bX_1)].%
\end{equation}
and %
\beq\label{zedef}
\zeta_i=H(\bX_i)-\theta, \;\;\;1\leq i \leq N;\quad  \cs^2(d) = \E \zeta_1^2.
\eeq
i.e., $\zeta_i$ are centered projections of the normalized $L_p$ norms of the differences $\mathbf X_i-\mathbf X_j$ onto the linear space of all measurable functions of $\mathbf X_i$.

\begin{theorem}\label{th2} Define the quantities %
\begin{equation}\label{e:def_a(x)}
 a(x)=(2\log x)^{1/2}\quad\mbox{and}\quad b(x)=2\log x +\tfrac{1}{2}\log \log x-\tfrac{1}{2} \log \pi.
\end{equation}   If  $H_0$ holds, and Assumptions \ref{as2},  \ref{as4}, and \ref{as5} are satisfied, then with $\cs(d)$ as in \eqref{zedef} and the choice $w(t)= (t(1-t))^{1/2}$,  as $\min\{N,d\}\to\infty$,
\begin{align}\label{th2-1}
P\biggl\{\frac{a(\log N)N^{1/2}}{2\cs(d)}~\! T_{N,d} \leq x + b(\log N)\biggl\} \to \exp\left(-2e^{-x}\right)
\end{align}
for all $x$. Furthermore,
\beq\label{th2-3}
\cs^2(d)=\frac{\sigma^2}{4d}+o\left(\frac{1}{d}\right),\;\;\;\;d\to\infty.
\eeq
\end{theorem}

Theorems \ref{th1} and \ref{th2} together provide the high-dimensional theoretical foundation for asymptotic tests based on $T_{N,d}$, encompassing essentially all possible weight functions of practical interest under the weak-type dependence context of Assumption \ref{as4}.   Upon selection of $w(t)$, an asymptotic size $\alpha$ test can be conducted simply by rejecting $H_0$ if $T_{N,d}$ exceeds the $1-\alpha$ quantile of the corresponding limit, provided a consistent estimate of $\sigma$ is available (on estimating $\sigma$, see Section \ref{s:implement}.)  %

\begin{Remark}{\rm 
The processes $V_{N,d},Z_{N,d}$ underlying $T_{N,d}$ are clearly highly dependent since they are both computed from the same sequence of data.  Under $H_0$, our proofs reveal the joint weak convergence of the $\mathbb R^2$-valued process $\sigma^{-1}(Nd)^{1/2}\big(V_{N,d}(t), Z_{N,d}(t)\big)^\top$ in the space $\cD[0,1]$ of $\mathbb R^2$-valued c\`adl\`ag functions on [0,1] to the process 
$$
\big(B(t), (1-2t)^{1-\beta}B(t)\big)^\top, \quad 0\leq t \leq 1,
$$
where $B(t)$ is a standard Brownian bridge, which ultimately drives the asymptotic behavior of $T_{N,d}$ under $H_0$.}
\end{Remark}

In the next section, we demonstrate the test can also be used under certain types of strong coordinate dependence as well.

\subsubsection{Strong dependence.} \label{s:strongdep}
Next we consider the setting when the coordinates are potentially strongly dependent, in the sense that they are discretely sampled observations of random functions $Y_i=\{Y_i(t), 0\leq t \leq 1\}$, $i=1,\ldots,N$, i.e.,
$$
\bX_i=(Y_i(t_1), Y_i(t_2), \ldots, Y_i(t_d))\;\;\mbox{with some }\; 0<t_1<t_2<\ldots <t_d\leq 1.
$$
Below is our main assumption in this setting.
\begin{assumption}\label{as6} \;(i) $Y_1,\ldots,Y_N$ are independent and identically distributed and have continuous sample paths with probability 1,\\
(ii)  $d \max_{1\leq i \leq d}(t_i-t_{i-1})\to 1,$ $t_0=0, t_d=1$\\
(iii)\; $E \int_0^1|Y_1(t)|^{r}dt<\infty$ for some $r \geq\max\{4,p\}$,  \\
(iv)\; and with some $\nu>2$
$$
E\left|\frac{1}{d}\sum_{j=1}^d |X_{1,j}|^p-\int_0^1|Y_1(t)|^pdt\right|^\nu\to 0.
$$
\end{assumption}
The projections $H(\mathbf x)$ in \eqref{zedef} can also be approximated with functionals of $Y_i(t)$. For any measurable function $g$, we define %
$$
\mathcal H(g)=E\left(\int_0^1 \left| Y_1(t)-g(t)\right|^pdt\right)^{1/p}
$$
whenever the expectation exists and is finite.  Next we define the asymptotic variance in the context of strongly dependent coordinates.  Let
$$
\gamma^2=\mbox{\rm var}\left( \mathcal H(Y_1)\right).
$$
Note $\gamma^2<\infty$ under Assumption \ref{as6}(iii). Since we normalize with $\gamma^2$, we naturally require

\begin{assumption}\label{as7}$ \gamma^2>0.$
\end{assumption}

Theorems \ref{th3} and \ref{th4}, stated next, are our main results in the context  of Assumption \ref{as6}.

\begin{theorem}\label{th3} Assume that $H_0$ holds and Assumptions \ref{as6} and \ref{as7} are satisfied. If $I(w,c)<\infty$ is finite for some $c>0$, then
\begin{align}\label{th3-1}
\gamma^{-1}N^{1/2}T_{N,d}\stackrel{\cD}{\to}\sup_{0<t<1}\frac{|B(t)|}{w(t)}
\end{align}
\end{theorem}
\medskip

The next statement is based on  the choice of weight function $w(t) = (t(1-t))^{1/2}$ and  is the analog of Theorem \ref{th2} in our strong dependence framework.   %
\begin{theorem}\label{th4} %
If  $H_0$ holds and  Assumptions \ref{as6} and \ref{as7} are satisfied, then with ${\cs}(d)$ as in \eqref{zedef} and the choice $w(t) = (t(1-t))^{1/2}$, as $\min\{N,d\}\to\infty$,%
\begin{align}\label{th4-1}
P\biggl\{\frac{a(\log N)N^{1/2}}{2\cs(d)}~\! T_{N,d} \leq x + b(\log N)\biggl\} =\exp\left(-2e^{-x}\right)
\end{align}
for all $x$. Furthermore, ${\cs}^2(d) =\gamma^2/4 +o(1)$. %
\end{theorem}
Theorems \ref{th1} through \ref{th4} together illustrate our approach is suitable in a variety of distributional settings concerning the dependence of the coordinates.  Though the scaling factor for $T_{N,d}$ in Theorems \ref{th1} and \ref{th2} is different than in Theorems \ref{th3} and \ref{th4} by a factor of $d^{1/2}$, conveniently, in practice the same implementation can be used in either case, as we discuss in the next section.

\subsubsection{On implementation}\label{s:implement}%

Implementation of asymptotic tests based on Theorems \ref{th1}--\ref{th4} first requires consistent estimation of the asymptotic variances $\sigma^2$ and $\gamma^2$ and suitable normalization of the test statistics.  Below we describe one such approach commonly used for $U$-statistics via the jackknife (e.g., \citet{lee:1990}).   %

In what follows, let $U_{N,d}$ denote $ d^{-1/p}U_{N,d,1}(N)$,  and let  $U_{N-1,d}^{(-i)}$  denote the value analogous to $U_{N,d}$ but based on the values $\bX_1,\ldots, \bX_{i-1}, \bX_{i+1}, \ldots, \bX_N$; i.e.\ $\bX_i$ is left out. Define the so-called \textit{pseudo-observations} 
$$
\cU_i=NU_{N,d}-(N-1)U_{N-1,d}^{(-i)}
$$
and their corresponding average
$$
\bar{\cU}_N=\frac{1}{N}\sum_{i=1}^N\cU_i.
$$
The jackknife estimator for the variance is defined as
$$
\hat{\sigma}_{N,d}^2=\frac{1}{(N-1)}\sum_{i=1}^N(\cU_i-\bar{\cU}_N)^2.
$$
The next statement shows that the same test statistic may be used in practice irrespective of which set of assumptions among Sections \ref{s:weakdep} and \ref{s:strongdep} hold.%
\begin{Proposition}\label{p:jack1} Under either the conditions of Theorem \ref{th1} or of Theorem \ref{th3}, as $\min\{N,d\}\to\infty$,
$$ \frac{N^{1/2}}{\hat \sigma_{N,d}}T_{N,d}  \stackrel{\cD}{\to}\sup_{0<t<1}\frac{|B(t)|}{w(t)}.$$
\end{Proposition}

To replace the normalizations in the Darling--Erd\H{o}s--type results of Theorems \ref{th2} and \ref{th4}, we need an assumption on the rate of convergence of $\widehat \sigma^2_{N,d}$. %
\begin{Proposition}\label{p:jack2} Under the conditions of Theorem \ref{th2}, if
$$
d^{1/2}\left|\hat{\sigma}_{N,d} -2d^{1/2}\cs(d)\right|\log \log N\stackrel{P}{\to}0
$$
Then,
\beq \label{e:darling_jackknife}
P\biggl\{\frac{a(\log N)N^{1/2}}{\hat\sigma_{N,d}}~\! T_{N,d} \leq x + b(\log N)\biggl\} \to \exp\left(-2e^{-x}\right)
\eeq
for all $x$.  Under the conditions of Theorem \ref{th4}, the limit \eqref{e:darling_jackknife} holds for all $x$ if
$$
\left|\hat{\sigma}_{N,d}-\cs(d)\right|\log \log N\stackrel{P}{\to}0.
$$
\end{Proposition}
Thus, upon choosing the weight function $w(t)$, one can appropriately normalize $T_{N,d}$ using $\hat\sigma_{N,d}$ to obtain the same limit distribution under $H_0$ in either the weak- or strong- coordinate dependence case. This leads to a testing procedure that, with regards to its asymptotic size, remains indifferent conercerning which assumption among these two sets holds.

Critical values of the limiting test statistics in Theorem \ref{th1} and Theorem \ref{th3} may be obtained through various means. For instance, \citet{franke} provide a fast adaptive method to approximate the critical values for the supremum of the Brownian bridge with weight function $w(t)=(t(1-t))^\kappa$, $0\leq \kappa<1/2$. Selected critical values for the weighted Brownian bridge are also tabulated in \citet{olmo}.   
If desired, resampling methods can also be used to provide critical values; for instance, one can use the bootstrap as in \citet{goho} and \citet{huki}; permutation-based methods can be also used to obtain critical values (cf. \citet{anhu}), though at potentially at a substantial computational cost. This is discussed toward the end of Section \ref{s:MC:discussion}.

\subsection{Power.}\label{s:power}

Next we briefly discuss the behavior of the statistics under $H_A$. For simplicity we first consider the case of a single change point $(R=1)$ at location $k_1$. Power in both single and multiple change scenarios is examined numerically in Section \ref{sec-sim}.  %

The next result provides sufficient conditions for high-dimensional consistency of the asymptotic tests given in Section \ref{s:size}.  For the statements ahead, define
\begin{equation}\label{e:mu1,mu2}
\mu_1=E\|\bX_1-\bX_2\|_p, \quad \mu_2=E\|\bX_{k_1+1}-\bX_{k_1+2}\|_p
\end{equation}
and
\begin{equation}\label{e:mu12}
 \mu_{1,2}=E\|\bX_1-\bX_{k_1+1}\|_p.
\end{equation}
(N.b.: $\mu_1,\mu_2,$ and $\mu_{1,2}$ depend on $d$ and are allowed to depend on $N$.) Also, set 
\begin{equation}\label{e:def_rescaledT_Nd}
\widetilde T_{N,d} = \begin{cases}\displaystyle\frac{N^{1/2} T_{N,d}}{\cs(d)}& \text{if $I(w,c)<\infty$ for some $c$,}\\
\displaystyle \frac{a(\log N)N^{1/2}}{\cs(d)}~\! T_{N,d}- b(\log N)& w(t)=(t(1-t))^{1/2}.
\end{cases}
\end{equation}
Note that divergence of $\widetilde T_{N,d}$ implies the consistency  of the associated test in Theorems \ref{th1}--\ref{th4}.  We also denote $0<\eta<1$ as the break fraction, i.e., the change point $k_1$ is given as $k_1=\lf N\eta\rf$.
\begin{theorem}\label{th:power} Suppose $R=1$ and  $k_1=\lf N\eta\rf,$ for some fixed $0<\eta<1$. Let 
$$
\Delta_{N,d} =d^{-1/p}\max\{|\mu_1-\mu_2|,|\mu_1-\mu_{1,2}|,|\mu_2-\mu_{1,2}|\} %
$$
Then, we have
\begin{equation}\label{e:widetildeT_diverg}
\widetilde T_{N,d}\stackrel P \to \infty
\end{equation}
whenever  Assumptions \ref{as2}, \ref{as4} and \ref{as5} hold,  and
\begin{enumerate}[(i)]
\item $(Nd)^{1/2} \Delta_{N,d}\to\infty$, if $w$ is chosen such that $I(w,c)<\infty$ for some $c$,
\item $(Nd)^{1/2} \left(\log\log N\right)^{-1/2}\Delta_{N,d} \to \infty$, if $w(t)=(t(1-t))^{1/2}$.
\end{enumerate}
Under Assumptions \ref{as6}--\ref{as7}, \eqref{e:widetildeT_diverg} holds whenever
\begin{enumerate}[(i')]
\item $N^{1/2} \Delta_{N,d}\to\infty$, if $w$ is chosen such that  $I(w,c)<\infty$ for some $c$,
\item $N^{1/2} \left(\log\log N\right)^{-1/2}\Delta_{N,d}\to\infty$, if $w(t)=(t(1-t))^{1/2}$.
\end{enumerate} 
\end{theorem}
Observe that the hypotheses of Theorem \ref{th:power} require a stronger separation condition in the context of Assumptions \ref{as6}--\ref{as7} compared to the context of Assumptions \ref{as4}--\ref{as5}. For instance, when $\Delta_{N,d}$ is held fixed, the conditions in $(i)$ and $(ii)$ suggest larger values of $d$ lead to improved test power under Assumptions \ref{as4}--\ref{as5}, whereas this is not the case for conditions $(i')$ and $(ii')$ under Assumptions \ref{as6}--\ref{as7}.   This can roughly be seen as a reflection of the fact that a stronger signal is needed to overcome comparatively stronger coordinate-wise dependence.%

It is also worth noting that even if  the distributions change at location $k_1$, it is still possible that $\mu_1=\mu_2$ holds.  For example, under a location shift,  $\mu_1=\mu_2$, but $\mu_1\neq\mu_{1,2}$. Therefore, Theorem \ref{th:power} demonstrates our tests are consistent against a variety of alternatives, encompassing changes in location, scale, and in higher moments.  %

Though our main focus is testing, the behavior of our test statistics under the alternative can be used to estimate the change point location. Below we provide one such possibility that is consistent under mild additional conditions. Let $$Z_{N,d,0}(t)=t(1-t) d^{-1/p}\left[U_{N,d,3}(\lf Nt\rf)-U_{N,d,4}\right],$$
 and let
$$k_V= \argmax_{2<k<N-2}\big|V_{N,d}\big( \tfrac{k}{N}\big)\big|,\quad k_Z=\argmax_{2<k<N-2}\big|Z_{N,d,0}\big(\tfrac{k}{N}\big)\big|,$$
where the selection of $k_V,k_Z$ is arbitrary among the respective maximizers in the case of ties. Set $t_V=k_V/N$ and $t_Z=k_Z/N$.  Our estimator is then defined as
\begin{equation}\label{e:changepoint_location_estimator}
\widehat \eta_{N,d} = \begin{cases} t_V, & |V_{N,d}(t_V)|\geq  |Z_{N,d}(t_Z)|,\\
t_Z, &  |V_{N,d}(t_V)| <|Z_{N,d}(t_Z)|.
\end{cases} %
\end{equation}
The idea behind the estimator is as follows.  Though the behavior of $Z_{N,d}$ and $V_{N,d}$ is quite similar under $H_0$, depending on the type of alternative,  their paths under $H_A$ can be dramatically different. Selecting the maximizers $t_Z$ and $t_V$ of the unweighted $Z_{N,d,0}$ and $V_{N,d}$  is done at the first step so that $\beta$ does not influence the location of the estimated change.  Then, the change point location is chosen among these candidate locations based on the associated value of the original processes $Z_{N,d}$ and $V_{N,d}$.  (Note that, in principle, both $Z_{N,d}$ and $V_{N,d}$ may exceed the critical value: in these cases, the break location is simply selected as the larger among the two.)

The following theorem establishes the consistency of $\widehat \eta_{N,d}$.

\begin{Proposition}\label{th:consistency}
Under the hypotheses of Theorem \ref{th:power}, for $\widehat \eta$ as in \eqref{e:changepoint_location_estimator}, we have
\beq\label{estcon}
\hat{\eta}_{N,d}\stackrel{P}{\to}\eta
\eeq
in the setting of Assumptions \ref{as2}, \ref{as4} and \ref{as5} whenever
\begin{enumerate}[(i)]
\item  $\mu_{1,2}\gg \max\{\mu_1,\mu_2, d^{1/p-1/2}N^{-1/2}\}$, %
\item $|\mu_{1}-\mu_2| \gg \max\{|\mu_{1,2}-\mu_1|,|\mu_{1,2}-\mu_2|\}\gg d^{1/p-1/2}N^{-1/2}$.
\end{enumerate}
In the setting of  Assumptions \ref{as5}-\ref{as6}, \eqref{estcon} holds whenever
\begin{enumerate}[(i')]
\item$\mu_{1,2}\gg \max\{\mu_1,\mu_2, d^{1/p}N^{-1/2}\}$,
\item $|\mu_{1}-\mu_2| \gg \max\{|\mu_{1,2}-\mu_1|,|\mu_{1,2}-\mu_2|\}\gg d^{1/p}N^{-1/2}$.%
\end{enumerate}
\end{Proposition}

The conditions for consistency in Proposition \ref{th:consistency} reflect two qualitatively different  types of possible changes: the divergence $(i)$ (and $(i')$) can occur due to a location shift, and the divergence $(ii)$ (and $(ii')$) can result from scale changes, among other possibilities (in particular, the divergence in $(ii)$ cannot occur due to a location shift).  %

\begin{Remark} \rm
Upon rejection of $H_0$, since $V_{N,d}$ is invariant under location shifts, it may be desirable to know which among $V_{N,d}$ and $Z_{N,d}$ contributed to rejection of $H_0$. By virtue of the joint maximization of these statistics through $T_{N,d}$, one can check which among  $\sup_{0<t<1}|V_{N,d}(t)|/w(t)$ and $\sup_{0<t<1}|Z_{N,d}(t)|/w(t)$ has exceeded the critical value of the test without increasing the overall type I error rate.
\end{Remark}

The estimation and testing procedure above can be extended naturally and readily to estimate multiple changes via recursive procedures such as binary segmentation, which we illustrate in Section \ref{sec-sim}.  Though in principle our estimation procedure can potentially be improved  further via wild binary segmentation \cite{fryzlewicz:2014}  or possibly through customized methods that make separate use of $Z_{N,d}$ and $V_{N,d}$,  for the sake of illustration and simplicity we focus on the case of ordinary binary segmentation in our numerical study.

\medskip
\section{Simulation study}\label{sec-sim}
To evaluate the numerical performance of our procedure, we examine its behavior in simulations based on the choice of %
the weight function
\beq\label{e:weightfunction}
w(u)=\frac{1}{(u(1-u))^\kappa}\quad\;\;0\leq \kappa\leq 1/2,
\eeq
for various values of $\kappa$.  In all settings, unless otherwise stated, we use $\beta=0.9$.  Each reported estimate is based on 2000 independent realizations. The asymptotic variance estimates used in each case are based on the jackknife method described in Section \ref{s:implement}.

\subsection{Size}  We examine empirical size of our test at the sample sizes $N=50,100,250,500$ with $d=N$ and $d=2N$ where $\mathbf X_i$ are independent copies of $\mathbf X=(X_1,\ldots,X_n)$, in the following settings:
\begin{enumerate}
\item $\mathbf X \sim N(0,\mathbf I_d)$;
\item AR(1) coordinates: $X_{i} = \phi X_{i-1} + \varepsilon_i$, with $\phi=0.9$ and $\varepsilon_i\stackrel{iid}\sim N(0,1)$. 
\item $\mathbf X\sim\text{Multinomial}(5d,\mathbf p)$, with $\mathbf p=(p_1,\ldots,p_d)$, where $p_i=(1/i) \big(\sum_{j=1}^d 1/j\big)^{-1}$.
\end{enumerate}
For brevity we display results only for $p=1$, though other cases of $p=2,3$ performed similarly in our simulations. Below are tables for nominal sizes $0.01$ and $0.05$. %
\begin{table}[!ht]\small
\begin{center}
\scalebox{0.8}{
\small
	\begin{tabular}{c c|| c c c c  | c c c c }
	\hline
		  \multicolumn{10}{c}{\footnotesize Nominal size $0.01$}  \\
		\toprule
		& & \multicolumn{4}{c|}{$N=d$}   & \multicolumn{4}{c}{$N=d/2$} \\
		 \cline{3-10}
		&$\kappa$ & 50 & 100 & 250 &500 & 50 & 100 & 250 &500 \\ 
		\hline\hline
		\multirow{4}{*}{Ex $(1)$} &0& .007&    .008 &    .008&     .010&    .004&    .006&    .006&      .007\\
		& .2 & .005  &   .005  &   .009  &  .009   &  .005    & .007  &  .011    &  .008\\
		& .4 &  .013   &   .010    & .007   & .011   &  .015  &  .010 &  .009 &     .012\\
		& .45 &  .024   & .020 &   .021   & .017 &   .017  &   .019  &  .016   &  .011\\ \hline
		\multirow{4}{*}{Ex (2)} &0& .003    & .003  &   .007   &  .010   & .005  &  .010 &   .006   & .007 \\
		& .2 & .004  &   .005 &   .009  &  .006 &   .004 &    .009   & .011  &  .008 \\
		& .4 &  .016  &  .011  &  .010   & .009  &  .014   &  .009  &   .012   &  .009 \\
		& .45 &  .014  &  .017  &  .017  &  .015  &   .023 &   .029 &   .022  &  .020\\ \hline
		\multirow{4}{*}{Ex (3)}&0 & .007  &  .007 &   .009  &   .008  &   .004   &  .005  &   .005 &   .006 \\
		& .2 &  .006   &  .006  &   .008  &  .007  &   .006  &    .010  &  .0075  &   .006 \\
		& .4 &   .025   & .014  &  .011 &   .012  &  .024  &  .019  &  .011  &    .010 \\
		& .45 &  .030  & .027  &  .010  &   .016 &  .034 &   .017   &  .018  &  .016 \\ \hline
		\toprule
	\end{tabular}}\ 
	\scalebox{0.8}{
\small
\begin{tabular}{c c|| c c c c  | c c c c }
	\hline
		  \multicolumn{10}{c}{\footnotesize Nominal size $0.05$}  \\
		\toprule
		& & \multicolumn{4}{c|}{$N=d$}   & \multicolumn{4}{c}{$N=d/2$} \\
		 \cline{3-10}
		&$\kappa$ & 50 & 100 & 250 &500 & 50 & 100 & 250 &500 \\ 
		\hline\hline
		\multirow{4}{*}{Ex $(1)$} & 0 & .026 & .032 & .039 & .048 & .028 & .038 & .047 & .046\\
		& .2 & .036 &   .044&   .040&     .042&    .037&     .034&    .043&    .044\\
		& .4 &  .064 &  .054 &  .044&     .051&    .053&     .058&    .050&    .044\\
		& .45 &  .065& .069 &  .064&     .053&    .061&     .069&    .079&    .069\\ \hline
		\multirow{4}{*}{Ex (2)} &0 &.027 &    .033 &   .036  &  .043  &   .023 &    .035 &   .042  &   .036 \\
		& .2 & .037  & .033  &   .039 &   .045  &   .030  &  .040  &   .050  &   .049 \\
		& .4 &  .060   &  .058   &  .049 &   .047 &  .060  &  .063  &  .047 &   .042 \\
		& .45 & .059  &   .067 &   .061  &    .06   &  .064  &  .071   &  .070  &   .062 \\ \hline
		\multirow{4}{*}{Ex (3)}&0 & .034  &  .033   & .039  &  .041   & .024   & .034   &  .036  &   .033 \\
		& .2 &  .044 &   .038 &   .049  &  .046  & .029 &    .034  &   .035   &  .047 \\
		& .4 &  .070   &  .059   & .058  &  .052 &   .069  &   .058  &   .058  &  .059 \\
		& .45 &  .072 &   .067 &   .061  &  .069  &  .064  &  .064   & .066   &  .064 \\ \hline
		\toprule
	\end{tabular}}
	\end{center}
		\centering
	\caption{Empirical size for nominal size $\alpha=0.01$ and $\alpha=0.05$ tests with $p=1$ in Examples $(1)$--$(3)$.}
	\label{tab2}
\end{table}
In general, the test tends to be more aggressive  as $\kappa$ increases, particularly in the case of $\kappa=0.45$ when it is consistently oversized. (Since the convergence rate for $\kappa=0.5$ is particularly slow, we do not recommend use of $\kappa=0.5$ at this sample size.) Generally the size approximation is closest to the nominal value for the case $\kappa=0.2$ and $\kappa=0.4$ for values of $N\geq 100$.  There is little difference between the settings with $d=2N$ and $d=N$, or between examples (1)--(3).

\subsection{Power.} We study power and estimation performance in a variety of settings.  
 For comparison against other methods suited for general changes, we compare our method with
\begin{itemize}
\item (E-div.): the nonparametric E--divisive method \cite{matteson:james:2014} via the \texttt{ecp} package in \texttt{R}.  This method involves choice of a parameter $\alpha\in(0,2)$; in our simulations we consider ordinary energy distance ($\alpha=1$).  Unless otherwise indicated we use default settings for the \texttt{e.divisive()} method.
\item (MEC): The graph-based max-type edge count method (\citet{chu:chen:2019}) via the \texttt{gSeg} package in \texttt{R}.  This method is suitable for a variety of datatypes and is demonstrated in their work to have some advantages in high dimensions. Following their simulation examples, we apply their method using the $k$-minimum-spanning tree graph based on Euclidean distances with $k=5$.
\end{itemize}
In all of our simulations in this section, we consider tests at the 5\% significance level.

\subsubsection{Single change setting.}\label{s:singlechange}We first examine the effect of the various types of distributional changes for a single change point in high dimensions with $N=d$ at location $\lfloor k_1/N\rfloor =\eta \in\{0.5,0.75\}$. In what follows, we let $\mathbf Y(\phi)=(Y_1,\ldots,Y_d)^\top$ denote vector of  AR(1) coordinates, where $Y_i=\phi Y_{i-1} + \varepsilon_i$ and $\varepsilon_i\stackrel{iid}\sim N(0,1)$.  Also, below, $t(\nu)$ denotes a Students' $t$ random variable with $\nu$ degrees of freedom. For the single change-point setting, we consider:
\begin{enumerate}\setcounter{enumi}{3}
\item Location change:  $\mathbf X_1\stackrel{d}= \mathbf Y(\phi)$ and $\mathbf X_{k_1+1}\stackrel d =\mathbf Y(\phi)+\boldsymbol \mu$, with $\phi=0.5$, $\boldsymbol \mu=(0.2,\ldots,0.2)^\top$ %
\item Covariance change: $\mathbf X_{1}\stackrel{d}= \mathbf Y(\phi)$, and $\mathbf X_{k_1+1}\stackrel{d}= \mathbf Y(\phi')$, with $\phi=0.5$, $\phi'=0.55$. (Note this amounts to a roughly 7.5\% gain in variance in addition to correlation changes.)
\item Tail change: $\mathbf X_{1}\sim N(0,\mathbf I_d)$, and $\mathbf X_{k_1+1}=(X_{k_1+1,1},\ldots X_{k_1+1,d})$ with $X_{k_1+1,i}\stackrel{iid}\sim t(\nu)/\sqrt{\var(t(\nu))}$, where $\nu=7$. (Note the first three moments remain constant throughout.)%
\end{enumerate}

For  E-div, when $N\leq 100$, we set the minimum cluster size to 10; otherwise we leave it as its default value (30).  For MEC, we use the function \texttt{gSeg1()}, as is recommended in the package documentation (\citet{gSeg}) for the single change-point setting.

\begin{table}[!ht]\small
\scalebox{0.9}{
	\begin{tabular}{c c | c c  c  c c| c c  c c c c|}
		\toprule
		& & \multicolumn{5}{c|}{$\eta=0.5$}  &  \multicolumn{5}{c}{$\eta=0.75$}   \\
		\cline{3-12}
		& $N=d$ &  $p=1$ & $p=2$ & $p=3$ & E-div. & MEC & $p=1$ & $p=2$ & $p=3$ & E-div. & MEC \\
		\hline\hline
		\multirow{5}{*}{Ex (4)\hspace{-2ex}}  & 50 & .059   &  .045   &  .055& {\bf .210} & .135 & .057  &  .052 &   .049 & .113  & \textbf{.117} \\
		&100 & .094  &  .071   &  .093& {\bf .620}& .277 &.057 &      .049    & .053& \textbf{.381} & .201\\
		&150 &   .455 &   .446  &  .442 & {\bf .985}& .553&  .058&     .054 &   .065 & \textbf{.901} & .376\\ 
		&200 &  .937 &    .938   & .923 & {\bf .999}& .826& .076 &    .075 &   .068 & \textbf{.998}& .599\\ 
		&250 &  {\bf 1.00 } &      {\bf 1.00}  &   .998 & .991 &.996 & .114   & .101  &  .112& \textbf{1.00}& .817\\ \hline
\multirow{5}{*}{Ex (5)\hspace{-2ex}}  
		& 50 & .090  & \bf .092 &   .087  & .045 & .073& .092  &   .078   &.078& .052 & .076 \\
		&100 & .274   & \bf .291 &   .268 & .069 &  .085 & \textbf{.228}&   .197  & .213 & .052&.069 \\
		&150 &  .598  &  \bf .617  &  .605 & .043& .112 & \textbf{.492} &    .452  &  .478 & .040&.110 \\ 
		&200 &  \bf.884  &  .878&   .878 & .048&  .141 & .755  &  \textbf{.760} &    .742& .041 & .120\\ 
		&250 & \bf .984  &   .982  &  .984 & .042& .207& .907   &  \textbf{.927}     & .925 & .053 & .203\\ \hline
\multirow{5}{*}{Ex (6)\hspace{-2ex}}  & 50 & .087   & .041   &  .064 & .041& \bf.208 & .075  & .067    & .094& .048 & \textbf{.186}\\
		&100 & .258     & .060  &  .144 & .035& \bf .300 &  .227   & .071 &   .204 & .052& \textbf{.337}\\
		&150 & \bf .561  &  .060 &    .347 & .042& .396 & \textbf{.445}    &  .070     &.373& .057& .365\\ 
		&200 & \bf .847  &   .055 &   .624 & .041& .454 & \textbf{.720}    & .083    &.625& .048 & .426 \\ 
		&250 &\bf  .971  &   .066  &  .864  &.045 & .392&  \textbf{.910}&     .092 &   .812 & .037 & .467\\ \toprule
	\end{tabular}
	}
		\centering
	\caption{Power in the single change-point settings of Examples $(4)$--$(6)$ at the changepoint locations $\eta=0.5$ and $\eta = 0.75$.  $T_{N,d}$  is calculated based on $\beta=0.9$ and $w(t)$ as in \eqref{e:weightfunction} with $\kappa=0.4$. We consider $p=1,2,3$ in \eqref{e:Ustats} (corresponding to three leftmost columns in each group). Boldface numbers  indicate the method with the highest power in each row.}
	\label{tab3}
\end{table}

  In general we see that E-div performs extremely well under a mean change as in setting (4), though MEC and our method still have reasonably good power by comparison for $N=d=200$ and larger  when $\eta=0.5$.  Rather surprisingly, both E-div and MEC are severely underpowered in setting (5), and our method substantially outperforms both comparison methods in high dimensions. In setting (6), the MEC method has good power for smaller $d$, but ultimately in high dimensions our method outperforms both MEC and E-div, provided $p\neq 2$.  Note that in (6), both the E-div method and our approach with $p=2$ have severe lack of power, in agreement with the phenomenon concerning limitations of Euclidean energy distance-related metrics  pointed out in \citet{chakraborty:zhang:2021}.  Since our method with $p=2$ is roughly similar to the Euclidean energy distance, this is somewhat expected.

In general our method has its highest power when $\eta=0.5$, as is expected and is common in most change-point methods. When $\eta=0.75$, at this sample size our approach displays a drop in power for mean changes compared to Ediv and MEC, but retains reasonably good power for settings (6) and (7). Though power certainly expected to decrease further in these settings as $\eta$ moves toward the boundary, this is partly compensated by the choice of weight function $w(t)$ with $\kappa=0.4$.  This likely can be further mitigated by increasing $\kappa$, though potentially at the expense of some control over size in moderately sized samples. %

\subsubsection{Multiple change setting.} In practice, the number of change points $R$ and their locations are unknown.  For illustration, below we examine the effect of multiple changes in settings with a random number  $\mathcal R\sim1+\text{Poisson(1)}$ of change points at locations $\eta_i=\lfloor k_i/N\rfloor=i/(R+1)$. For simplicity of exposition, we consider ``switching'' scenarios that alternate between two distributional behaviors  analogous to Examples (4)--(6). %
Specifically, for every integer $j$ with $0\leq 2j \leq \mathcal R$, recalling $k_0=1$, $k_R=N$, we consider

\begin{enumerate}\setcounter{enumi}{6} 
\item Location change:  $\mathbf X_{k_{2j}}\stackrel{d}= \mathbf Y(\phi)$, and if $2j<\mathcal R$,  $\mathbf X_{k_{2j}+1}\stackrel d =\mathbf Y(\phi)+\mu$, with $\phi=0.5$, $\boldsymbol \mu=(0.2,\ldots,0.2)^\top$,  %
\item Covariance change: $\mathbf X_{k_{2j}}\stackrel{d}= \mathbf Y(\phi)$, and if $2j<\mathcal R$, $\mathbf X_{k_{2j}+1}\stackrel{d}= \mathbf Y(\phi')$, with $\phi=0.5$, $\phi'=0.55$.
\item Tail change: $\mathbf X_{k_{2j}}\sim N(0,\mathbf I_d)$, and if $2j<\mathcal R$, $\mathbf X_{k_{2j}+1}=(X_{k_{2j}+1,1},\ldots X_{k_{2j}+1,d})$ with $X_{k_{2j}+1,i}\stackrel{iid}\sim t(\nu)/\sqrt{\var(t(\nu))}$, $\nu=7$. (As in Example (6), the first three moments remain constant throughout.)
\end{enumerate}
In scenarios $(7)$--$(9)$, each replication has its own independent realization of $\mathcal R\in\{1,2,\ldots\}$. In our simulations, we estimate the unknown number of change points $\mathcal R$ and change point locations recursively at the 5\% level of significance using ordinary binary segmentation. To assess estimation performance of both the estimated number of changes $\widehat{ \mathcal R}$ and the estimated change point locations in high dimensions, in the table below, we we consider the high dimensional setting of $N=d$, with $N=300,400,500,600$, and report
\begin{itemize}
\item The mean and median of the {error} $\widehat {\mathcal R} -\mathcal R$
\item The average Rand index (RI) and average adjusted Rand index (ARI).
\end{itemize}
The RI and ARI are measures of  agreement between two clusterings of data. In each realization, we compute the RI and ARI based on the estimated partition dictated by the estimated change point location(s) compared with the true partition for that realization consisting of $\mathcal R+1$ clusters. Values near 1 indicate strong agreement with the clusterings, i.e., between the estimated number of change points and their true values, in addition to estimated change point locations and their true values.  Values near 0 indicate strong disagreement.  For more details on the Rand and adjusted Rand indices, we refer to \citet{rand:1971} and \citet{morey:agresti:1984}.

\begin{table}[!ht]\small
	\begin{tabular}{r c | c c  c c| c  c c c| c  c c c }
		\toprule
		& & \multicolumn{4}{c|}{$T_{N,d}$, ~~ Bin. Seg.} & \multicolumn{4}{c|}{E-div.}& \multicolumn{4}{c}{MEC, Bin. Seg.}\\
		\cline{3-14}\rule{0pt}{2.5ex}
		& & \multicolumn{2}{c}{$\widehat{\mathcal R}-\mathcal R$}  &\multicolumn{2}{|c|}{Rand idx.}&  \multicolumn{2}{c}{$\widehat{\mathcal R}-\mathcal R$}  &\multicolumn{2}{|c|}{Rand idx.}  &  \multicolumn{2}{c}{$\widehat{\mathcal R}-\mathcal R$}  &\multicolumn{2}{|c}{Rand idx.}  \\
		& $N=d$ &  $Q_2$ &  Mean &\multicolumn{1}{|c}{ RI}& ARI & $Q_2$ & Mean &\multicolumn{1}{|c}{ RI}& ARI & $Q_2$ & Mean &\multicolumn{1}{|c}{ RI}& ARI  \\
		\hline\hline
		\multirow{4}{*}{Ex (7)\hspace{-1ex}} &300 & -2  & -1.54 &  .524  & .311& 0 & 0.04 & .985 &\bf .967   &-2 &-1.27& .610 &.425\\
		&400 & -2& -1.50 &  .559 &    .369 &0 & 0.04 & .991 & \bf.978& 0 &-0.99 &.684 & .524\\
		&500 & -2& -1.42 & .578 &    .398 &0 & 0.05 & .994 & \bf.986 & 0 &-0.99 &.684 & .524\\ 
		&600 & -2& -1.33& .604 &    .430 & 0 & 0.05 & .995 & \bf.989  & 0 &-0.26 & .847 & .758\\\hline
		\multirow{4}{*}{Ex (8)\hspace{-1ex}}   
		&300 & -2  & -0.62 &  .800  & \bf.675& -2 & -1.91&  .384& .023&  -2 &  -1.51& .499 & .233    \\
		&400 & 0& -0.09 &  .928 &    \bf.862 & -2 & -1.89 & .390&  .025 & -2 & -1.34 & .565 &.350 \\
		&500 & 0& 0.12 & .969 &   \bf .930 &  -2 &-1.89 & .391& .029 & 0 & -0.96& .666 &  .500  \\
		&600 & 0& 0.20& .978 &    \bf.949 &-2 &-1.98 &  .384&.031 & 0 & -0.52&  .787 &.673 \\\hline
		\multirow{4}{*}{Ex (9)\hspace{-1ex}} 
		&300 & -2  & -0.68 &  .784  & \bf.652& -2 & -1.98& .377 & .021  & -2 &-1.70& .411 & .065  \\
		&400 & 0& -0.16&  .918 &    \bf.845 & -2 & -1.93 & .384 & .019& -2 &-1.71 & .431 &.116  \\
		&500 & 0&  0.05 & .962 &    \bf.918 & -2 & -1.90  & .389 & .019 & -2& -1.59 & .441 &.134\\
		&600 & 0& 0.14& .976 &   \bf .944 &-2 &  -1.91& .383 & .023 & -2 &-1.53 & .471 & .180  \\\toprule

	\end{tabular}
		\centering
	\caption{Estimation accuracy with a random $\mathcal R\sim 1 +\text{Poisson}(1)$ number of change points in the setting of Examples $(7)$--$(9)$. $Q_2$ denotes the median of $\widehat {\mathcal R}-\mathcal R$. $T_{N,d}$ (leftmost group) is calculated based on $\beta=0.9$, $w(t)$ as in  \eqref{e:weightfunction} with $\kappa=0.4$, and in \eqref{e:Ustats} we take $p=1$. Boldface numbers indicate the method with the highest average ARI in each row.}
	\label{tab4}
\end{table}
For E-div, we use the same settings as in the Section \ref{s:singlechange}.  For MEC, we continue to use the \texttt{gSeg1()} function with the same settings to accommodate the possibility of only one change point in each realization; repeated application of this method is recommended in the \texttt{gSeg} package documentation as an approach for detecting multiple changes. We implement ordinary binary segmentation to facilitate direct comparison with applying binary segmentation to our procedure. 

In general our method has good estimation accuracy in settings (8) and (9). We see that in all settings, behavior analogous to the single change-point setting emerges: except for location changes, our method has dramatically higher performance by comparison to Ediv; the same is true in setting (9) compared to MEC, and in setting (8), our method outperforms both approaches, but MEC retains reasonable estimation performance.   For the multiple mean-change setting (7), E-div still dominates for these parameter choices, followed by MEC, and by comparison to the single change-point setting, our estimation performance deteriorates though it retains moderate levels of ARI.  %

\subsection{Discussion} \label{s:MC:discussion}
Among the three types of changes considered, our test has reasonably good performance across an array of alternatives, and can have high power and good estimation accuracy in high-dimensional settings settings where other nonparametric methods are relatively powerless or have difficulty correctly estimating change points.  Though no single method is expected to uniformly perform best against every type of alternative, our procedure still displays moderate power and estimation accuracy in high dimensions even when it is outperformed for location-only alternatives, whereas comparison methods can be severely underpowered for covariance or tail changes, indicating our approach may be a suitable for use when little knowledge about the anticipated change points is to be assumed.  In settings in which location changes are of primary interest, however,  likely a procedure suited specifically to location changes would be more appropriate.

	In addition, though not explored in-depth in this study,  there may be possible advantages of increasing $p$. As evident from Table \ref{tab3}, larger values of $p$ appear to lead to a decrease in power for the given alternatives in settings (4)--(6).  However, for sparse alternatives in which only a small portion of coordinates change, increasing $p$ may provide some benefit; unreported simulation studies suggest increasing $p$ can help with power against sparse location alternatives, in particular, but less so for sparse covariance or sparse tail changes.
	
	Further, unreported simulation studies suggest in most settings it is best to take $\beta$ close to 1; though control over size can degrade when making it extremely close to 1 (e.g., say 0.99) in moderate  $(N\leq 500)$ sample sizes; in general recommend that $\beta\leq 0.9$ when $N\leq 500$, and $\beta \leq 0.5$ for small sample sizes.    $\beta$ also influences the balance of the two statistics $V_{N,d}$ and $Z_{N,d}$ underlying our method; for smaller $\beta$, $Z_{N,d}$ is somewhat deprioritized and our test sensitivity is altered in a nontrivial way, partly becoming less sensitive to location alternatives. On the other hand, increasing $\beta$ over larger samples can lead to increases in power against location alternatives.  To minimize the choice of tuning parameters in practice, we recommend in general setting $\beta=0.9$ which seems to work reasonably well in a variety of scenarios.
	
	Our procedure can likely improved in a variety of ways without making significant changes to the underlying method.  For instance, estimation performance in the multiple change point setting is expected to improve with wild binary segmentation  (Fryzlewicz, 2014).  Over moderate and smaller-sized samples in particular,  power can likely also be improved though refined normalization of $Z_{N,d}(t)$, e.g., replacing $(|1-2t| +N^{-1/2})^{-\beta}$ by $(|1-2t| + h_N(t))^{-\beta}$ or similar, for a function $h_N(t)$ satisfying $h_N(1/2)\sim N^{-1/2}$ but equal to 1 outside a neighborhood of $t=1/2$.  Also note in the multiple change setting,  in \eqref{e:changepoint_location_estimator}, each of $t_V$ and $t_Z$ may concentrate around different change points, but only one of $t_V$ and $t_Z$ are chosen as the segmentation point, which may result in losses in power in the subsequent iteration of the binary segmention.  This could potentially be improved by running parallel segmentations that take into account which statistic(s) exceed the critical value,  which in principle would also provide a more detailed analysis of the observed sequence.
	
	Lastly, in small samples, in lieu of  using our asymptotic approach, a standard permutation-based test for $T_{N,d}$ can instead be used to retain control over size.  However our method as well as E-div have quadratic complexity as the sample size increases, and permutation-based approaches require repeated rearrangements of $O(N^2)$ distances for each permutation that can be computationally burdensome when $N$ is large (even when a small number of random permutations are chosen, c.f. \citet{biau:bleakley:etal:2016}).  Our asymptotic approach has the benefit of avoiding this problem and makes our test suitable for testing potentially longer sequences compared with resampling approaches.
	\counterwithin{figure}{section}

\section{Application: mentions of  U.S.\ governors on Twitter }\label{sec-appl}
\counterwithin{figure}{section}
\counterwithin{table}{section}
We illustrate our method through an application involving Twitter data concerning mentions of U.S. governors.   The data was collected using the full--archive tweet counts endpoint in the Twitter Developer API\footnote{\url{https://developer.twitter.com/en/docs/twitter-api/tweets/counts/introduction}}. The API allows retrieval of the count by day of tweets matching any query from the complete history of public tweets. %
Note that when using the full-archive tweet counts endpoint in the Twitter API, the filters \texttt{-is:retweet -is:reply -is:quote} were applied to remove retweets, replies and quote tweets respectively. %

  Our dataset consists of matching queries that reference any of the 50 U.S. governors from 1/1/21 to 12/31/21, resulting in a series of length and dimension $(N,d) = (365,50)$. Only governors holding office on the date 12/31/21 were included in this dataset; the full list of queried names is given in Table \ref{t:govnames}.

To accommodate variability in the total number of daily mentions among all governors, which range from roughly 1,000 to 20,000 mentions in a given day, a subset of $m=500$ observations on each day were randomly sampled without replacement, resulting in conditionally multivariate hypergeometric observations $\mathbf X_i$ with parameters $m=500$, $\mathbf r_i$, where the vector ${\mathbf r}_i =(r_{1,i},\ldots,r_{50,i})$ contains the observed number of daily mentions for each governor based on the random selection for that day. Test results did not substantially change for other choices of $m\leq 1000$, and repeated tests for different subsets of $m$ gave relatively consistent results.  
\begin{center}
\begin{table}[h]
\scalebox{0.8}{
\begin{tabular}{| m{15cm} |}\hline\
\Centering U.S. Governors holding office on 12/31/21
\\
\hline\hline
{\small Greb Abbott, Charlie Baker, Andy Beshear, Kate Brown, Doug Burgum, John Carney, Roy Cooper, Spencer Cox, Ron DeSantis, Mike Dewine, Doug Ducey, Mike Dunleavy, John Bel Edwards, Tony Evers, Greg Gianforte, Mark Gordon, Michelle Grisham, Kathy Hochul, Larry Hogan, Eric Holcomb, Asa Hutchinson, David Ige, Jay Inslee, Kay Ivey, Jim Justice, Laura Kelly, Brian Kemp, Ned Lamont, Bill Lee, Brad Little, Dan McKee, Henry McMaster, Janet Mills, Phil Murphy, Gavin Newsom, Kristi Noem, Ralph Northam, Mike Parson, Jared Polis, J.B. Pritzker, Tate Reeves, Kim Reynolds, Pete Ricketts, Phil Scott, Steve Sisolak, Kevin Stitt, Chris Sununu, Tim Walz, Gretchen Whitmer, Tom Wolf} \\\hline
\end{tabular}
}
\caption{\label{t:govnames} Governor names used for query in Twitter data collection}
\end{table}
\end{center}

For our tests and estimation, we use $p=1$ and $\beta=0.9$ in accordance with favorable performance of these choices revealed in Section \ref{sec-sim}.  After a change is detected, we estimate the location of change with $\hat{\eta}_{N,d}$, the estimator of \eqref{e:changepoint_location_estimator} with weight function $w(t)=(t(1-t))^\kappa$ with for various $\kappa$ and continue to estimate changes via binary segmentation and tests were repeated in each subsegment until failure to reject the null hypothesis.  Figure \ref{f:gov_changepoints} contains detected change points for $p=1$, which were identical for choices of $\kappa=0.0,0.1,0.2,0.25,0.4$ and $\beta =0.9$ at significance level $0.01$.

Several detected changes apparently coincide with important dates or events in the U.S.\ news cycle.   For instance, the date 9/12/21 (adjacent to 9/11) is detected as a change point, as is the date 11/25/21, coinciding with the weekend following Thanksgiving; some news events are apparent from other dates. For instance, the date 5/21/2021 coincides with several news stories following Texas Governor Greg Abbott's signing of the controversial Texas Heartbeat Act into law\footnote{\url{https://en.wikipedia.org/wiki/Texas_Heartbeat_Act}}; the date 8/1/2021 coincides with a news cycle immediately following Florida Governor Ron Desantis'  executive order on 7/30/2021 concerning masking.\footnote{\url{https://www.flgov.com/2021-executive-orders/}} Interestingly, unreported results from repeated analysis on this dataset across various values of $\beta=0,0.1,0.2,0.3,0.4$ and $\kappa=0,0.2,0.4$ reveals the dates 5/21/21, 8/1/21, and 9/12/2021 are all detected in nearly every case.

\begin{figure}[!ht]
\scalebox{0.5}{\begin{forest} for tree={grow=south,circle,draw,minimum size=1ex,inner sep=2mm,s sep=3mm, l sep+=-3ex}
[Nov. 27
	[Aug. 1
	 	 [~Feb. 2~\ ~
		 	[,no edge, draw=none][,no edge, draw=none][,no edge, draw=none][,no edge, draw=none][,no edge, draw=none][,no edge, draw=none][,no edge, draw=none][,no edge, draw=none] [ Apr. 6
								[,no edge, draw=none][,no edge, draw=none][Mar. 25][,no edge, draw=none][,no edge, draw=none][,no edge, draw=none][May 21
									[May 9][,no edge, draw=none][,no edge, draw=none]
							  	       ][,no edge, draw=none][,no edge, draw=none]
							    ]
		 ][,no edge, draw=none][,no edge, draw=none][,no edge, draw=none][,no edge, draw=none]
	[ Oct. 15
		[ Sep. 12
			[Aug 30] [,no edge, draw=none]
		] [,no edge, draw=none]
		][,no edge, draw=none][,no edge, draw=none][,no edge, draw=none][,no edge, draw=none][,no edge, draw=none]
		][,no edge, draw=none][,no edge, draw=none][,no edge, draw=none][,no edge, draw=none][,no edge, draw=none][,no edge, draw=none][,no edge, draw=none][Dec. 15][,no edge, draw=none][,no edge, draw=none][,no edge, draw=none]
]
\end{forest}}

\caption{\label{f:gov_changepoints} Change-point detection with size $0.01$ tests using $p=1$ for Governor dataset from January 1st, 2021 to December 31st, 2021.  Bisection conducted until failure to reject null hypothesis.  In the figure above, the top of each tree is the first estimated change point; subsequent estimated change points after bisection appear along each respective branch in the order of detection, displayed chronologically left--to--right.}
\end{figure}
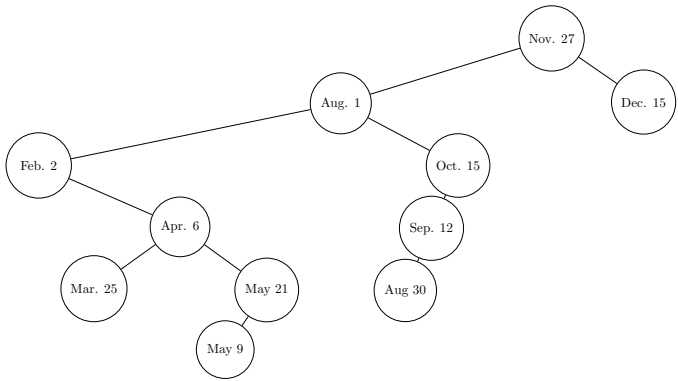

\section{Conclusion}

In this paper, we construct an asymptotic testing procedure suitable for high-dimensional  settings based on combining two maximally selected statistics stemming from $L_p$ norms.   Our method is theoretically supported in high dimensional asymptotic settings and displays convenient limit behavior leading to a straightforward asymptotic test with readily available critical values.  We have demonstrated our test has reasonably good power against a variety of alternatives, and has especially high power when there is no change in location by comparison to other nonparametric methods, making it suited to scenarios when little a priori knowledge is available about the type of possible change points or when location changes are not expected.

	An underlying principle of our approach lies in combining two separate statistics -- each measuring different aspects of the data -- that behave similarly under $H_0$, leading to tractable and convenient limit behavior,  but behave differently under $H_A$, thereby providing increased power.  This general principle can likely be used in other high-dimensional change-point methods, or to develop more discerning change-point tests and estimation procedures. %
	We leave this open as a topic for future research.
\appendix
\section{Some examples}\label{app-ex}
In this section, we provide some examples in which the assumptions in Section \ref{s:theoreticalresults} are satisfied. Throughout, $C>0$ denotes a generic constant independent of $N,d$ whose value may change line-to-line. %

\begin{example}[Independent coordinates]\label{exe1}  {\rm Assume $X_{1,1}, X_{1,2}, \ldots, X_{1,d}$ are independent, and that \eqref{as4-3} holds.  For simplicity write  $\bX_1-\bX_2=(Z_1, Z_2, \ldots, Z_d)^\T$.  Using Rosenthal's inequality \citep[p.~59]{petrov:1995} we get for all $\beta\geq 2$
\begin{align*}
E\Bigg|&\sum_{j=1}^d \big(|Z_j|^p -E|Z_j|^p\big)\bigg|^\beta\\
&\leq C\Bigg(\sum_{j=1}^d E\big||Z_j|^p -E|Z_j|^p\big|^\beta%
+\bigg[\sum_{j=1}^d\var \big(|Z_j|^p\big)\bigg]^{\beta/2}\Bigg),
\end{align*}
where $C>0$ depends only on $p$. Similarly,
\begin{align*}
E\bigg|&\sum_{j=1}^d\big[ g_j(X_{1,j})  -E  g_j(X_{1,j}) \big]\bigg|^\beta\\
& \leq C\Bigg( \sum_{j=1}^d  E\big| g_j(X_{1,j})  -E  g_j(X_{1,j})\big|^\beta + \bigg[\sum_{j=1}^d\var\big(g_j(X_{1,j}) \big) \bigg]^{\beta/2}\Bigg).
\end{align*}
Note for all $\beta\geq 1$, $\E |g_j(X_{1,j})|^\beta \leq C \E |X_{1,j}|^{p\beta}$.  Hence conditions \eqref{as4-1} and \eqref{as4-2} in Assumption \ref{as4} are satisfied if
$$
\limsup_{d\to \infty}\frac{1}{d}\sum_{j=1}^d \E |X_{1,j}|^{2p}<\infty
$$
and
$$
\limsup_{d\to \infty}\frac{1}{d^{\alpha/2}}\sum_{j=1}^d\left[E|X_{1,j}|^{p\alpha} +(E|X_{1,j}|^p)^\alpha\right]<\infty.
$$
}
\end{example}
Next we extend Example \ref{exe1} to dependent coordinates.%
\begin{example}[Linear process coordinates]\label{exe2}{\rm
Suppose for each $d$, $\mathbf X_i=(X_{i,1},\ldots, X_{i,d})^\top$ are given by $X_{i,\ell} = \sum_{j=-\infty}^\infty a_{\ell-j} \varepsilon^{(i)}_j$, where $\varepsilon^{(i)}_j$ are independent identically distributed variables with $E \varepsilon _j^{(i)} = 0$, $\E (\varepsilon_i^{(i)})^2=1$, $\E |\varepsilon_1|^{p\alpha}<\infty$, and the coefficients $a_k$ satisfy
\beq\label{exe2-2}
 |a_{k}|\leq c_1\exp(-c_2|k|)
\eeq
with some $0<c_1, c_2<\infty$. First we show \eqref{as4-1} holds, i.e., that %
\beq\label{exe2-22}
E\bigg|\sum_{j=1}^d\left[ |Z_j|^p-E|Z_j|^p    \right]\bigg|^\alpha\leq Cd^{\alpha/2}
\eeq
where $(Z_1, Z_2, \ldots, Z_d)^\T=\bX_1-\bX_2$. %
 For notational simplicity let $\varepsilon_j = \varepsilon_j^{(1)}-\varepsilon_j^{(2)}$, and for $c_3>0$, define %
$$
\bar{Z}_j=\sum_{|j-\ell|\leq c_3\log d} a_{\ell-j}\varepsilon_\ell+\sum_{|j-\ell|>c_3\log d}a_{\ell-j}\widetilde \varepsilon_\ell,
$$
where the sequence $\{\widetilde \varepsilon_j\}_{j=-\infty}^\infty$ is an independent copy of $\{\varepsilon_j\}_{j=-\infty}^\infty$.  Let $q\geq 1$ be such that $q^{-1}+ r^{-1} =1$, and take any $s>\frac{1}{r}$. The decay of $a_{j,k}$ implies we can choose  $c_3=c_3(\alpha)$ large enough so that 
\begin{align*}
\E |Z_j-\bar{Z}_j|^{r} &= E \bigg| \sum_{|j-\ell|>c_3\log d}a_{\ell-j}\widetilde \varepsilon_\ell \Big|^{r}\\
& \leq \bigg(\sum_{|j-\ell|>c_3\log d}|j-\ell|^{sq}|a_{\ell-j}|^q \bigg)^{\frac{r}{q}}  \bigg( \sum_{|j-\ell|>c_3\log d}|j- \ell|^{-sr} E| \widetilde \varepsilon _\ell |^r \bigg) \\
& \leq C \bigg(\sum_{|j-\ell|>c_3\log d}|j-\ell|^{sq}|\exp\{-qc_2|j-\ell|\} \bigg)^{\frac{r}{q}}\\
& \leq C d^{-\alpha/2}.\notag
\end{align*}%
Using the inequality $||x|^p-|y|^p|\leq C_p(|x|^{p-1}+|y|^{p-1})|x-y|$, this gives
\begin{align}\label{e:xbarj_minus_xj}
E\bigg|\sum_{j=1}^d\big|Z_j|^p-|\bar{Z}_j\big|^p\bigg|^\alpha&\leq CE\Bigg(\sum_{j=1}^d(|Z_j|^{p-1}+|\bar{Z}_j|^{p-1})|Z_j-\bar{Z}_j| \Bigg)^\alpha \\
&\leq  Cd^{\alpha-1} \sum_{j=1}^dE\Big((|Z_j|^{p-1}+|\bar{Z}_j|^{p-1})|Z_j-\bar{Z}_j| \Big)^\alpha \\
& 
\leq  d^{\alpha-1} C\sum_{j=1}^d \big(E|Z_j|^{p \alpha }\big) ^{\frac{p-1}{p}}\big(\E |Z_j-\bar{Z}_j|^{p\alpha} \big)^{1/p}\leq C d^{\alpha/2}.
\end{align}
Thus, it suffices to establish \eqref{exe2-22} for the variables $\bar Z_j$ in place of $Z_j$. Note that by definition, $\bar{Z}_1, \bar{Z}_2, \ldots, \bar{Z}_d$ are $c_3\log d$--dependent random variables. For simplicity, for each $d$ we set $\bar Z_\ell=0$ whenever $\ell>d$.  Now define $n_k= (k-1)\lfloor (\log d)^2\rfloor$, $\zeta_j=|\bar Z_j|^p-E|\bar Z_j|^p$. %
Let
$$
Q_{k,1}=\sum_{\ell=n_{k-1}+1}^{n_k}\zeta_\ell,\;\;\;k=1, 3, \ldots, k^*_1, \qquad
Q_{k,2}=\sum_{\ell=n_{k-1}+1}^{n_k}\zeta_\ell,\;\;\;k=2,4,\ldots, k^*_2,
$$
where $k_1^*,$ and $k_2^*$ are the smallest odd and even integers, respectively, such that $n_{k_i^*}\geq d$. (Note the quantities $Q_{k^*_1,1}$ and $Q_{k^*_2,2}$ respectively may contain fewer than $n_{k^*_1}-n_{k^*_1-1}$ and $n_{k^*_2}-n_{k^*_2-1}$ terms.) By construction, the variables $ Q_{k,1}, k=1,3,\ldots, k^*_1$ are independent and similarly  $ Q_{k,2}$, $k=2,4,\ldots,k^*_2$ are independent. Thus, using Rosenthal's inequality again, we obtain
\begin{equation}\label{e:ros_ex_2}
E\bigg|\sum_{k=1,3,\ldots,k^*_1} Q_{k,1} \bigg|^\alpha \leq c_9\left( \sum_{k=1,3,\ldots,k^*_1}\E |Q_{k,1}|^\alpha+ \Big( \sum_{k=1,3,\ldots,k^*_1} \E (Q_{k,1})^2 \Big)^{\alpha/2} \right)
\end{equation}
 and similarly for $Q_{k,2}$.  We proceed to bound $E |Q_{k,1}|^\alpha$ and $E |Q_{k,1}|^2$ separately.  For  $E |Q_{k,1}|^\alpha$, we have
\begin{equation}\label{e:Q_k_alphabound}
 E |Q_{k,1}|^\alpha \leq (n_k-n_{k-1})^{\alpha-1}\sum_{\ell=n_{k-1}+1}^{n_k} \E |\zeta_\ell|^\alpha \leq C(n_k-n_{k-1})^\alpha \leq C (\log d)^{2\alpha}.
\end{equation}
For $E |Q_{k,1}|^2$, we need a sharper bound.  For each $j$, set
$$
\bar \varepsilon_\ell^{(j)} = \begin{cases}
\varepsilon_\ell  & |\ell-j|\leq c_3 \log d \\  \widetilde \varepsilon_\ell  & |\ell-j|> c_3\log d\\
\end{cases}%
$$
so that $\bar{Z}_j=\sum_{\ell=-\infty}^{\infty} a_{\ell-j}\bar \varepsilon_\ell^{(j)}$.
Now, for each $1\leq \ell,j\leq d$, $j\neq \ell$, set $I_{j,\ell}=\big(j-\lfloor |j-\ell|/2\rfloor,j+\lfloor |j-\ell|/2\rfloor\big]$, and define
$$
\bar Z_{j,\ell} = \sum_{i\in I_{j,\ell}}a_{i-\ell}\bar \varepsilon_i^{(j)}+\sum_{i\notin I_{j,\ell}} a_{i-\ell}\bar \varepsilon_i^{(j,\ell)}
$$
where $\bar \varepsilon_i^{(j,\ell)}$, $i\not\in I_{j,\ell}$  are independent copies of $\bar \varepsilon_i^{(j)}$. %
By construction, for each pair $(j,\ell)$, $j\neq\ell$, we have $ \bar Z_{\ell,j}\stackrel{\cD}= \bar Z_j$, and the variables $ \bar Z_{\ell,j}$ and $ \bar Z_{\ell,j}$ are independent, since $I_{j,\ell} \cap I_{\ell,j}=\emptyset$.  Further, if we define
$$\zeta_{j,\ell}= |\bar Z_{j,\ell}|^p -\E |\bar Z_{j,\ell}|^p%
$$
we clearly have, for some constants  $c_{11},c_{12}>0$ independent of $d$,
$$
\E|\zeta_j-\zeta_{j,\ell}|^2 = \sum_{i \notin I_{j,\ell}} a_{i,j}^2 \leq  c_{11}\exp(-c_{12}|j-\ell|).
$$
Therefore, from the decomposition $E |Q_{k,1}|^2 = \sum_{\ell=n_{k-1}+1}^{n_k}\E \zeta_\ell^2 + 2\sum_{n_{k-1}+1\leq j<\ell\leq n_k}\E\zeta_\ell\zeta_j $, using $\E \zeta_{j,\ell}\zeta_{\ell,j}=0$, we obtain
\begin{align*}
\Big|\sum_{n_{k-1}+1\leq j<\ell\leq n_k}\E \zeta_\ell\zeta_j \Big| &= \Big|\sum_{n_{k-1}+1\leq j<\ell\leq n_k}\E \big[(\zeta_\ell-\zeta_{\ell,j })\zeta_j \big] + \E \big[(\zeta_j-\zeta_{j,\ell})\zeta_{\ell,j })\big]\Big|\\
&\leq \sum_{n_{k-1}+1\leq j<\ell\leq n_k}\Big(\big(\E (\zeta_\ell-\zeta_{\ell,j })^2\E\zeta_j^2) \big)^{1/2} + \big(\E (\zeta_j-\zeta_{j,\ell})^2\E\zeta_{\ell,j }^2\big)^{1/2} \Big)\\
& \leq C (n_{k}-n_{k+1}).
\end{align*}
Thus, 
\begin{equation}\label{e:Q_k_varibound}
\E |Q_{k,1}|^2 \leq  C (n_{k}-n_{k+1}).
\end{equation}
Combining \eqref{e:Q_k_alphabound} and \eqref{e:Q_k_varibound} with \eqref{e:ros_ex_2}, we obtain $E\big|\sum_{k=1,3,\ldots,k^*_1} Q_{k,1} \big|^\alpha \leq C d^{\alpha/2}$.  The same arguments apply to $Q_{k,2}$, which establishes \eqref{exe2-22}, i.e.  \eqref{as4-1} holds. 

Turning to \eqref{as4-2}, observe
$$
|g_j(x)-g_j(y)|\leq C(|x|^{p-1}+|y|^{p-1})|x-y|.
$$
Thus, \eqref{as4-2} %
can be established using the same arguments leading to \eqref{exe2-22} with minor adjustments.

 For  \eqref{as4-3}, note the sequence $g_1(X_1),g_2(X_2),\ldots$ is stationary and thus  $\ca(p)= E g_1(X_1)$ and the existence of $\tau^2$ follows, showing that Assumption \ref{as4-1} holds.
} 
\end{example}

\begin{example}[Multinomial coordinates]\label{exe3} {\rm We assume that $\bX_1$ has a multinomial distribution with parameter $(\gamma_1/d, \gamma_2/d, \ldots, \gamma_d/d, Cd)$, where $C$ is a positive integer, $0<c_1\leq \gamma_i\leq c_2, 1\leq i \leq d, \gamma_1+\gamma_2+\ldots+\gamma_d=d$. Under these assumptions, the coordinates of $\bX_1$ are approximately independent Poisson random variables (cf.\  \citet{mcd} and \citet{deh})
and the conditions of Theorems \ref{th1} and \ref{th2} are satisfied.
}
\end{example}

\medskip
\section{Preliminary lemmas} \label{sec-pr1}

In all asymptotic statements, $O_P(1)$ and $o_P(1)$ denote terms which are bounded in probability and tend to zero in probability, respectively, in the limit $\min\{N,d\}\to\infty$. Throughout, we continue to write $C>0$ to denote a generic constant, independent of $N,d$ whose value may change line-to-line.

Recall the  independent and identically distributed random variables $\zeta_i, 1\leq i \leq N$ defined in \eqref{zedef}.  Let
\begin{align}\label{psize}
\psi_{i,j}=\left(\frac{1}{d}\sum_{\ell=1}^d|X_{i,\ell}-X_{j,\ell}|^p\right)^{1/p}-\theta.%
\end{align}
First we provide a bound for the distance between the underlying U-statistics and their projections in terms of the (dimension-dependent) quantities $\E\psi^2_{1,2}$ and $\E\zeta_1^2$.
\begin{lemma}\label{lem-1} For each pair $1\leq i\neq j\leq N$, let $y_{i,j}= \psi_{i,j}-(\zeta_i+\zeta_j)$. Under either the conditions of Theorem \ref{th1} or \ref{th2} or under Theorem \ref{th3} and \ref{th4}, %
for any $\bar{\alpha}>0$
\begin{align}\label{lem-11}
\max_{1\leq k\leq N}\frac{1}{k^{1+\bar{\alpha}}}\bigg| \sum_{1\leq i<j\leq k}y_{i,j}  \bigg|=O_P\left( \left(E\psi^2_{1,2}+E\zeta_1^2  \right)^{1/2}\right),
\end{align}
\begin{align}\label{lem-12}
\max_{1\leq k\leq N-1}\frac{1}{(N-k)^{1+\bar{\alpha}}}\bigg|  \sum_{k+1\leq i<j\leq N}y_{i,j}  \bigg| =O_P\left( \left( E\psi^2_{1,2}+E\zeta_1^2  \right)^{1/2}\right),
\end{align}
and for any $ \beta\geq 0$,
\begin{align}\label{lem-13}
\max_{1\leq k\leq  N-1}\displaystyle\frac{\left( \left|1-\displaystyle\frac{2k}{N}\right| +\displaystyle\frac{1}{\sqrt N}\right)^{-\beta}}{[k(N-k)]^{\frac{1}{2}+ \bar\alpha }}&\Bigg| \sum_{i=1}^k\sum_{j=k+1}^N y_{i,j} \Bigg|\notag \\
&=O_P\left( \big( \log(N)  N^{ -\bar\alpha} + \log^2(N)  N^{\frac \beta2-2\bar\alpha)}\big)  \left(E\psi^2_{1,2}+E\zeta_1^2  \right)^{1/2}\right).
\end{align}
\end{lemma}
\begin{proof} We write
$$
\sum_{1\leq i<j\leq k}[\psi_{i,j}-(\zeta_i+\zeta_j)]=\sum_{j=2}^k\xi_j,
$$
where
$$
\xi_j=\sum_{\ell=1}^{j-1} [\psi_{\ell, j}-(\zeta_\ell+\zeta_j)],\quad\;\;2\leq j \leq N;
$$
for convenience we set $\xi_1=0$. Using the definition of $\zeta_i$, we have
$$
E\xi_i=0,\;\;E\xi_i\xi_j=0,\;\;\mbox{if}\;\;i\neq j\quad\mbox{and}\quad E\xi_i^2\leq  9i\left(E\psi^2_{1,2}+2E\zeta_1^2  \right).
$$
Hence by Menshov's inequality (cf.\ Billingsley, 1968, p.\ 102)
\begin{align}\label{men}
E\max_{1\leq k \leq m}\left(\sum_{i=1}^k \xi_i \right)^2\leq (\log_2(4m))^2\sum_{i=1}^m  9i\left(E\psi^2_{1,2}+2E\zeta_1^2  \right)
\leq C m^2(\log m)^2\left(E\psi^2_{1,2}+E\zeta_1^2  \right).
\end{align}
Using \eqref{men} we get 
\begin{align*}
P&\bigg\{ \max_{1\leq k \leq N}\frac{1}{k^{1+\bar\alpha}}\bigg| \sum_{i=1}^k\xi_i   \bigg|>x  \left(E\psi^2_{1,2}+E\zeta_1^2  \right)^{1/2}  \bigg\}\\
&\leq P\left\{ \max_{1\leq j \leq \log N}\max_{e^{j-1}\leq k \leq e^j}\frac{1}{k^{1+\bar\alpha}}\bigg| \sum_{i=1}^k\xi_i   \bigg|>x  { \left(E\psi^2_{1,2}+E\zeta_1^2  \right)^{1/2}  }  \right\}\\
&\leq \sum_{j=1}^{\log N}P\bigg\{ \max_{e^{j-1}\leq k \leq e^j}\frac{1}{k^{1+\bar\alpha}}\bigg| \sum_{i=1}^k\xi_i   \bigg|>x
  \left(E\psi^2_{1,2}+E\zeta_1^2  \right)^{1/2}   \bigg\}\\
&\leq \sum_{j=1}^{\log N}P\bigg\{ \max_{e^{j-1}\leq k \leq e^j}\bigg| \sum_{i=1}^k\xi_i   \bigg|>x e^{(j-1) (1+\bar\alpha)}   \left(E\psi^2_{1,2}+E\zeta_1^2  \right)^{1/2}   \bigg\}\\
&\leq \frac{C}{x^2}\sum_{j=1}^{\log N}e^{-2j(1+ \bar\alpha)}e^{2j}j^2\\
&\leq \frac{C}{x^2},
\end{align*}
completing the proof of \eqref{lem-11}. By symmetry, \eqref{lem-11} implies \eqref{lem-12}.  We now show \eqref{lem-13}.  First, by Lemma \ref{l:mensh}, for each fixed $1\leq m_1<m_2\leq N-1$,
\begin{align}\label{e:maximal_inequ}
E\max_{m_1\leq k \leq m_2}\bigg(\sum_{i=1}^k\sum_{j=k+1}^N y_{ij}\bigg)^2\leq \big[\log_2(4m_1)\log_2(4(N-m_1)\big]^2 m_2(N-m_1)\left(E\psi^2_{1,2}+E\zeta_1^2  \right).
\end{align}
 For any $x>0$, let $x'= x\log(N) N^{ -\bar\alpha} $.  Using \eqref{e:maximal_inequ}, over the range $1\leq k < N/4$, for all large $N$, we have $\big( \big|1-\frac{2k}{N}\big| +\frac{1}{N}\big)^{-\beta}\leq 2$, and thus
\begin{align}
P&\left\{ \max_{1\leq k \leq N/4}\displaystyle\frac{\left( \left|1-\displaystyle\frac{2k}{N}\right| +\displaystyle\frac{1}{\sqrt N}\right)^{-\beta}}{[k(N-k)]^{\frac{1}{2} + \bar{\alpha}}}\Bigg|  \sum_{i=1}^k\sum_{j=1}^{N-k} y_{ij} \Bigg|>x'  \left(E\psi^2_{1,2}+E\zeta_1^2  \right)^{1/2}  \right\} \notag\\
&\leq P\bigg\{ \max_{1\leq j \leq \log (N/4)}\max_{e^{j-1}\leq k \leq e^j-1}\frac{1}{k^{\frac{1}{2}+\bar{\alpha}}}\bigg|  \sum_{i=1}^k\sum_{j=1}^{N-k} y_{ij} \bigg|>C x'   N^{\frac{1}{2}+\bar\alpha}\left(E\psi^2_{1,2}+E\zeta_1^2  \right)^{1/2}   \bigg\} \notag\\
&\leq \sum_{j=1}^{\log (N/4)}P\bigg\{ \max_{e^{j-1}\leq k \leq e^j-1}\bigg|  \sum_{i=1}^k\sum_{j=1}^{N-k} y_{ij} \bigg|>x' C e^{(j-1) (\frac{1}{2} + \bar \alpha)}N^{\frac{1}{2} + \bar \alpha}  \left(E\psi^2_{1,2}+E\zeta_1^2  \right)^{1/2}  \bigg\} \notag\\
&\leq \frac{C}{(x')^2}\sum_{j=1}^{\log (N/4)}e^{- (1 + 2\bar \alpha)j }N^{-(1+2\bar\alpha)} e^{j}(N-e^{j-1})  \big[\log (4(N-e^{j-1}))\log (4 e^{j})\big]^2 \notag\\
&\leq \frac{C}{(x')^2} [\log(N)]^2 N^{-2\bar\alpha}  \sum_{j=1}^{\log (N/4)}j^2 e^{-j 2\alpha} \notag\\
&\leq \frac{C}{x^2}.\label{e:Op_nobeta}
\end{align}
Over the range $N/4\leq k \leq N/2$, $[k(N-k)]^{\bar{\alpha}/2}\geq  C N^{\bar\alpha}$, and applying \eqref{e:maximal_inequ} again,
\begin{align}
P&\left\{ \max_{N/4 \leq k \leq N/2}\displaystyle\frac{\left( \left|1-\displaystyle\frac{2k}{N}\right| +\displaystyle\frac{1}{\sqrt N}\right)^{-\beta}}{[k(N-k)]^{\frac{1}2+\bar{\alpha}}}\left|  \sum_{i=1}^k\sum_{j=1}^{N-k} y_{ij} \right|>x  \left(E\psi^2_{1,2}+E\zeta_1^2  \right)^{1/2}  \right\}\notag\\
&\leq P\left\{  \max_{N/4 \leq k \leq N/2}\left|  \sum_{i=1}^k\sum_{j=1}^{N-k} y_{ij} \right|>x' C  N^{1+2\bar\alpha-\frac{\beta}{2}} \left(E\psi^2_{1,2}+E\zeta_1^2  \right)^{1/2}   \right\}\notag\\
& \leq \frac{C}{x^2}\big( N^{\beta-4\alpha}\big[\log N\big]^4\big)\label{e:Op_withbeta}
\end{align}

By \eqref{e:Op_nobeta} and \eqref{e:Op_withbeta},  
\begin{align*}
\max_{1\leq k\leq  N/2}\frac{\displaystyle \Big( \Big|1-\frac{2k}{N}\Big| +\frac{1}{\sqrt N}\Big)^{-\beta}}{[k(N-k)]^{\frac{1}{2}+\bar{\alpha}}}&\Bigg| \sum_{i=1}^k\sum_{j=k+1}^N y_{i,j} \Bigg|\notag \\
&=O_P\left( \Big( \log(N)  N^{-\bar \alpha } + [\log(N)]^2  N^{\frac{\beta}{2}-2\bar\alpha}\Big)  \left(E\psi^2_{1,2}+E\zeta_1^2  \right)^{1/2}\right).
\end{align*}
By symmetry, the same bound holds for the maximum taken over $N/2<k \leq N$,  completing the proof of \eqref{lem-13}.

\end{proof}

Next we consider approximations for the sums of the projections. Let
$$
\cs^2(d)=E\zeta_1^2\quad\quad\mbox{and}\quad \quad \cmm(d,\nu)=E|\zeta_1|^{\nu}.
$$
\begin{lemma}\label{lem-2} Suppose $ \cmm(d,\nu)<\infty$. If either the conditions of Theorem \ref{th1} or \ref{th2} or those of Theorem \ref{th3} and \ref{th4}, are satisfied, then for each $N$ and $d$ we can define independent Wiener processes $\{W_{N,d,1}(x), 0\leq x \leq N/2\}$ and $\{W_{N,d,2}(x), 0\leq x \leq N/2\}$ such that,
\beq\label{lem-21}
\sup_{1\leq x \leq N/2}\frac{1}{x^{\bar\alpha}}\left|\frac{1}{\cs(d)}\sum_{i=1}^{\lf x \rf}\zeta_i-W_{N,d,1}(x)  \right|=O_P(1),%
\eeq
and
\begin{align}\label{lem-22}
&%
\sup_{N/2\leq x \leq N-1}\frac{1}{(N-x)^{\bar\alpha}}\left|\frac{1}{\cs(d)}\sum_{i=\lf x \rf+1}^{N}\zeta_i-W_{N,d,2}(N-x)  \right|=O_P(1)%
\end{align}
for any {$\bar\alpha>\max\{1/4, 1/\nu\}$}. %
\end{lemma}
\begin{proof} Using the Skorokhod embedding scheme \citep[e.g.,][]{brei}  we can define Wiener processes $W_{N,d,3}(x)$ such that
$$
\sum_{i=1}^k\zeta_i=W_{N,d,3}\left(\ct_i\right),\qquad
\ct_i=\ccr_1+\ccr_2+\ldots +\ccr_i,
$$
where for each $N$ and $d$ the random variables $\ccr_1, \ccr_2. \ldots, \ccr_{N/2}$ are independent and identically distributed with
$$
E\ccr_1=\cs^2(d)\quad\quad\mbox{and}\quad\quad E\ccr_1^{\nu/2}\leq c_1\cmm(d,\nu).
$$
We can assume without loss of generality that $2<\nu<4$. Using the Marcinkiewicz--Zygmund and von Bahr--Esse\'en inequalities \citep[p. 82]{petrov:1995} we get
\begin{align*}
E\max_{1\leq j\leq k}\left|\sum_{i=1}^j(\ccr_i-\cs^2(d))\right|^{\nu/2}&\leq C k\left(E\ccr_1^{\nu/2}+(\cs^2(d))^{\nu/2}\right)\\
&\leq Ck\left( \cmm(d,\nu) +(\cs^2(d))^{\nu/2} \right).
\end{align*}
We get for all  $\nu/2<\lambda$ that 
\begin{align}
P&\left\{\max_{1\leq k \leq N/2}\frac{1}{k^{\lambda}}\left| \sum_{i=1}^{k}(\ccr_i-\cs^2(d))   \right|>x  \left(\cmm(d,\nu)+(\cs^2(d))^{\nu/2}\right)^{2/\nu} \right\}\notag\\
&\leq P\left\{\max_{1\leq j\leq \log (N/2)}\max_{e^{j-1}\leq  k \leq e^j}\frac{1}{k^{\lambda}}\left| \sum_{i=1}^{k}(\ccr_i-\cs^2(d))   \right|>x  \left(\cmm(d,\nu)+(\cs^2(d))^{\nu/2}\right)^{2/\nu}\right\}\notag\\
&\leq P\left\{\max_{1\leq j\leq \log (N/2)}\max_{e^{j-1}\leq  k \leq e^j}\left| \sum_{i=1}^{k}(\ccr_i-\cs^2(d))   \right|>x e^{(j-1){\lambda}} \left(\cmm(d,\nu)+(\cs^2(d))^{\nu/2}\right)^{2/\nu}\right\}\notag\\
&\leq \sum_{j=1}^{\log (N/2)}P\left\{\max_{e^{j-1}\leq  k \leq e^j}\left| \sum_{i=1}^{k}(\ccr_i-\cs^2(d))   \right|^{\nu/2}>x^{\nu/2} e^{(j-1){\lambda}\nu/2} \left(\cmm(d,\nu)+(\cs^2(d))^{\nu/2}\right)\right\}\notag\\
&\leq \frac{C}{x^{\nu/2}}\sum_{j=1}^{\log (N/2)}e^je^{-(j-1){\lambda}\nu/2} \notag\\
&\leq \frac{C}{x^{\nu/2}}.\label{e:t_k_close_to_ ks^2(d)}
\end{align}
For $s>0$, and any $c_1>0$, let
$$
h({s},d)=c_1{s}^{\lambda}\left(\cmm(d,\nu)+(\cs^2(d))^{\nu/2}\right)^{2/\nu}
$$
 The bound \eqref{e:t_k_close_to_ ks^2(d)} implies
$$\lim_{c_1\to\infty}\liminf_{N\to\infty}\P\Big\{\ct_{\lfloor x\rfloor} \in \big\{y>0: |y- x\cs^2(d)| \leq  h(x,d)\big\},~ 1\leq x \leq N/2\Big\} = 1.$$ In turn, this implies%
\begin{align*}
\lim_{c_1\to \infty}\liminf_{N\to\infty}P\biggl\{&\max_{1\leq x \leq N/2}\frac{1}{\cs(d)x^{\bar\alpha}}\left|  W_{N,d,3}\left( \ct_{\lf x\rf} \right)-W_{N,d,3}\left(\cs^2(d)x  \right)\right|\\
&\leq \sup_{1\leq x\leq N/2}\sup_{\{ y>0: |y-\cs^2(d)x|\leq h(x,d)\}}\frac{1}{\cs(d)x^{\bar\alpha}}\left|  W_{N,d,3}\left( y \right)-W_{N,d,3}\left(\cs^2(d)x  \right)\right|
\Biggl\}=1.
\end{align*}
By the scale transformation of the Wiener process
\begin{align}
\sup_{1\leq x\leq N/2}&\sup_{\{y>0:|y-\cs^2(d)x|\leq h(x,d)\}}\frac{1}{\cs(d)x^{\bar\alpha}}\left|  W_{N,d,3}\left( y \right)-W_{N,d,3}\left(\cs^2(d)x  \right)\right|\notag\\
&\stackrel{\cD}{=}
\sup_{1 \leq x\leq N/2}\sup_{\{z>0:|z-x|\leq h(x,d)/\cs^2(d)\}}\frac{1}{x^{\bar\alpha}}\left| W(z)-W(x)   \right|\notag\\
&\leq 
\sup_{1 \leq x\leq N/2}\sup_{\{z>0:|z-x|\leq c_2 x^\lambda \}}\frac{1}{x^{\bar\alpha}}\left| W(z)-W(x)   \right|,\label{e:ss_bound}
\end{align}
where we used that $\sup_{d\geq 1}\left(\cmm(d,\nu)+(\cs^2(d))^{\nu/2}\right)^{2/\nu}/\cs^2(d)\leq c_2<\infty$ for some $c_2>0$ by Lemma \ref{appe-mom}.  Recall \citep[e.g.,][p. 24]{csorgo:revesz:1981} for every $\epsilon>0$ there exists a $c_0(\epsilon)$ such that
$$
\P\Big\{\sup_{0 \leq x\leq T}\sup_{\{z>0:|z-x|\leq h \}}\left| W(z)-W(x)   \right|>h^{1/2} y\Big\} \leq \frac{c_0(\epsilon) T}{h} \exp\Big(-\frac{y^2}{2+\epsilon}\Big),
$$
for any $T>0$ and $0<h<T$. Therefore, for any $M>0$, $c>0$, and  $ \lambda<2\bar\alpha$, we have
\begin{align*}
\P\Big\{\sup_{1 \leq x\leq N}&\sup_{\{z>0:|z-x|\leq c x^\lambda \}}\frac{1}{x^{\bar\alpha}}\left| W(z)-W(x)   \right|> M \Big\}\\
& \leq \sum_{j=1}^{\log N }\P\Big\{\sup_{e^{j-1}\leq x\leq e^j}\sup_{\{z>0:|z-x|\leq c x^\lambda \}}\frac{1}{x^{\bar\alpha}}\left| W(z)-W(x)   \right|> M\Big\}\\
& \leq \sum_{j=1}^{\log N }\P\Big\{\sup_{e^{j-1}\leq x\leq e^j}\sup_{\{z>0:|z-x|\leq c x^\lambda \}}\left| W(z)-W(x)   \right|> Me^{\bar\alpha (j-1)}\Big\}\\
& \leq \sum_{j=1}^{\log N }\P\Big\{\sup_{0\leq x\leq e^j}\sup_{\{z>0:|z-x|\leq c e^{j\lambda} \}}\left| W(z)-W(x)   \right|> c^{-1}e^{-\bar\alpha}  (c e^{j\lambda/2})  \cdot M e^{j(\bar\alpha-\lambda/2)} \Big\}\\
& \leq C \sum_{j=1}^{\log N }  e^{j(1-\lambda)}\exp\Big(-c_3 M^2 e^{j(2\bar\alpha-\lambda)} \Big)\\
& \leq C \sum_{j=1}^{\infty }  e^{j(1-\lambda)}\exp\Big(-c_3 M^2 e^{j(2\bar\alpha-\lambda)} \Big)\\
&= C \mathcal G(M).
\end{align*}
Note $\mathcal G(M)<\infty$ since  $\lambda<2\bar\alpha$, and clearly $\mathcal G(M)\to0$ as $M\to\infty$.  Returning to \eqref{e:ss_bound}, this implies
$$
\sup_{1 \leq x\leq N/2}\sup_{\{z>0:|z-x|\leq c_2 x^\lambda \}}\frac{1}{x^{\bar\alpha}}\left| W(z)-W(x)   \right| = O_P(1),\quad N\to\infty,
$$
which establishes the result.

\end{proof}

\medskip
\section{Proofs of Theorems \ref{th1}--\ref{th4}}\label{sec-pr2}

\medskip
\noindent
{\bf Proof of Theorem \ref{th1}.} Let $2/N\leq t \leq 1-2/N$.  With $k=\lfloor Nt\rfloor$, and $k^* = N-k$, note
\begin{align}
\sigma^{-1}(Nd)^{1/2} V_{N,d}(t) &= N^{1/2}\frac{k}{N}\frac{k^*}{N} \frac{2 d^{1/2}}{\sigma}\bigg(\frac{1}{k(k-1)} \sum_{1\leq i<j\leq k}\psi_{i,j} -\frac{1}{k^* (k^*-1)} \sum_{k+1\leq i<j\leq N}\psi_{i,j}\bigg)\notag\\
& = N^{1/2}\frac{k}{N}\frac{k^*}{N} \frac{2 d^{1/2}}{\sigma}\bigg(\frac{1}{k} \sum_{i=1}^k\zeta_i  -\frac{1}{k^*} \sum_{i=k+1}^N \zeta_{i}\bigg) + R^V_{N,d}(t)\notag\\
& = N^{-1/2}\frac{2 d^{1/2}}{\sigma} \bigg( \sum_{i=1}^k\zeta_i  -\frac{k}{N}\sum_{i=1}^N \zeta_{i}\bigg) + R^V_{N,d}(t),\label{e:sum_of_iid_terms}
\end{align}
where, letting $y_{ij}= \psi_{i,j} -\zeta_i-\zeta_j$,
$$R^V_{N,d}(t)= N^{1/2}\frac{k}{N}\frac{k^*}{N} \frac{2 d^{1/2}}{\sigma}\Big(\frac{1}{k(k-1)} \sum_{1\leq i<j\leq k} y_{ij}-\frac{1}{k^* (k^*-1)} \sum_{k+1\leq i<j\leq N}y_{ij}\Big).$$ 
For $\bar\alpha>0$, observe
$$ %
N^{1/2-\bar{\alpha}}\frac{|R_{N,d}(t)|}{(t(1-t))^{\bar{\alpha}}}=  \frac{2 d^{1/2}}{\sigma}\Big(\frac{(1-t)^{1-\bar\alpha}}{k^{\bar\alpha}(k-1)} \sum_{1\leq i<j\leq k} y_{ij}-\frac{t^{1-\bar\alpha}}{(k^*)^{\bar\alpha} (k^*-1)} \sum_{k+1\leq i<j\leq N}y_{ij}\Big).
$$
Thus, if $0<\bar\alpha<1$, Lemmas \ref{lem-1} and \ref{appe-mom} imply
 $$N^{1/2-\bar{\alpha}}\sup_{2/N\leq t \leq 1-2/N}\frac{|R_{N,d}(t)|}{(t(1-t))^{\bar{\alpha}}}=O_P(1).$$
 Similarly, writing $g_N(t) =\big(|1-2t| +\tfrac{1}{\sqrt N}\big)^{-\beta}$, we have%
\begin{align}
\sigma^{-1}(Nd)^{1/2}Z_{N,d}(t) &= \frac{g_N(t)}{N^{3/2}} \frac{2d^{1/2}}{\sigma}\bigg(\sum_{i=1}^k\sum_{j=k+1}^N \psi_{i,j} -\frac{k k^*}{N^2}\sum_{i=1}^N\sum_{j=1}^N\psi_{i,j}\bigg)\notag\\
& = \frac{g_N(t)}{N^{3/2}}  \frac{2d^{1/2}}{\sigma}\bigg(\sum_{i=1}^k\sum_{j=k+1}^N (\zeta_i+\zeta_j)-\frac{k k^*}{N^2}\sum_{i=1}^N\sum_{j=1}^N(\zeta_i+\zeta_j)\bigg)+ R^Z_{N,d}(t)\notag\\
& =  g_N(t) \bigg(1-\frac{2k}{N}\bigg) N^{-1/2}\frac{2d^{1/2}}{\sigma}  \bigg( \sum_{i=1}^k\zeta_i  -\frac{k}{N}\sum_{i=1}^N \zeta_{i}\bigg)+ R^Z_{N,d}(t),  \label{e:Z_Nd_expansion}
\end{align}
where
$$R^Z_{N,d}(t)%
= \sigma^{-1}(Nd)^{1/2} t(1-t) \bigg(\frac{g_N(t)}{kk^*}\sum_{i=1}^k\sum_{j=k+1}^N y_{ij}-\frac{g_N(t)}{N^2}\sum_{i=1}^N\sum_{j=1}^Ny_{ij}\bigg).
$$
Now, for every $0<\delta<1/2$,
\begin{align*}
N^{1/2-\bar{\alpha}} &\frac{|R_{N,d}(t)^Z|}{(t(1-t))^{\bar{\alpha}}}\\
&=  \frac{d^{1/2}}{\sigma} \bigg(  (t(1-t))^{\frac{1}2-\delta}\frac{N^{\bar \alpha-2\delta}g_N(t)}{(kk^*)^{\frac{1}{2}+\bar\alpha-\delta }}\sum_{i=1}^k\sum_{j=k+1}^N y_{ij}-\frac{[t(1-t)]^{1-\bar\alpha} g_N(t)}{N^{\frac{3}{2}+\bar\alpha}}\sum_{i=1}^N\sum_{j=1}^Ny_{ij}\bigg).
\end{align*}
Applying Lemmas \ref{lem-1} and \ref{appe-mom}, provided $\bar\alpha>\delta$, we obtain
 $$N^{1/2-\bar{\alpha}}\sup_{2/N\leq t \leq 1-2/N}\frac{|R^Z_{N,d}(t)|}{(t(1-t))^{\bar{\alpha}}}=O_P\Big(N^{-\delta} \log N + (\log N)^2 N^{\frac \beta2-\bar\alpha}\Big) + O_P(N^{\frac \beta2-\bar \alpha-\frac{1}2}\big).$$ Now, let
\begin{align}\label{e:B_Nd}
B_{N,d}(t)=\left\{
\begin{array}{ll}
\displaystyle N^{-1/2}\left[W_{N,d,1}(Nt)-t\left(W_{N,d,1}(N /2) +W_{N,d,2}(N /2) \right)\right],\;\;\;0\leq t \leq 1/2
\vspace{.3cm}\\
\displaystyle N^{-1/2}\left[-W_{N,d,2}(N-Nt)+(1-t) \left(W_{N,d,1}(N /2)+W_{N,d,2}(N /2) \right) \right],\;\;\;1/2\leq t \leq 1.
\end{array}
\right.
\end{align}
The process $\{B_{N,d}(t), 0\leq t\leq 1\}$ is Gaussian, and a straightforward computation of its covariance function  shows it is a Brownian bridge in law for all $N$ and $d$.  By \eqref{e:sum_of_iid_terms} and \eqref{e:Z_Nd_expansion}, 
\begin{align*}
\sigma^{-1}(Nd)^{1/2}Z_{N,d}(t)-R^Z_{N,d}(t)&=(1-2t)g_N(t)\big[\sigma^{-1}(Nd)^{1/2}V_{N,d}(t)-R^V_{N,d}(t)\big].
\end{align*}
 Observe $2 d^{1/2}/\sigma = (1/\mathcal s(d))\big(1 + o(1)\big)$ by Lemma \ref{appe-mom}.  Thus, based on expressions \eqref{e:sum_of_iid_terms} and \eqref{e:Z_Nd_expansion}, using that $|(1-2t)g_N(t)|\leq 1$, Lemma \ref{lem-2} implies
\begin{align*}
\sup_{2/N\leq t \leq 1-2/N}&\frac{|\sigma^{-1}(Nd)^{1/2}Z_{N,d}(t)-R^Z_{N,d}(t) - g_N(t)(1-2t)B_{N,d}(t)|}{(t(1-t))^{\bar{\alpha}}}\\
&\leq \sup_{2/N\leq t \leq 1-2/N}\frac{g_N(t)|1-2t||\sigma^{-1}(Nd)^{1/2}V_{N,d}(t)-R^V_{N,d}(t) - B_{N,d}(t)|}{(t(1-t))^{\bar{\alpha}}}\\
&\leq \sup_{2/N\leq t \leq 1-2/N}\frac{|\sigma^{-1}(Nd)^{1/2}V_{N,d}(t)-R^V_{N,d}(t) - B_{N,d}(t)|}{(t(1-t))^{\bar{\alpha}}}\\
&=O_P(N^{\bar{\alpha}-1/2}) + o_P(1)
\end{align*} %
for all $\bar \alpha>\frac{1}{4}$.  Thus, putting together Lemmas \ref{lem-1} and \ref{lem-2}, taking $\bar\alpha$ and $\delta$ so that $\delta<\frac \beta2<\bar{\alpha}\leq 1/2$,
\begin{align}\label{pr1-2}
&\sup_{2/N\leq t \leq 1-2/N}\frac{1}{(t(1-t))^{\bar{\alpha}}}
\left|\sigma^{-1}(Nd)^{1/2}Z_{N,d}(t)- g_N(t)(1-2t)B_{N,d}(t)\right| \notag\\ 
&\quad  \quad +\sup_{2/N\leq t \leq 1-2/N}\frac{1}{(t(1-t))^{\bar{\alpha}}}
\left|\sigma^{-1}(Nd)^{1/2}V_{N,d}(t)-B_{N,d}(t)\right|=O_P(N^{\bar{\alpha}-1/2}) + o_P(1).
\end{align}
Using that $\big|g_N(t)(1-2t)- (1-2t)^{1-\beta }\big|\leq  C (N^{\frac{\beta-1}{2}}{\mathbbm 1}_{\{|u|\leq 2N^{-1/2}\}} + u^{-\beta} N^{-1/2}{\mathbbm 1}_{\{|u|> 2N^{-1/2}\}}) \leq C N^{\frac{\beta-1}{2}}$, %
$$
\sup_{2/N\leq t \leq 1-2/N}\frac{1}{(t(1-t))^{\bar{\alpha}}} \left|g_N(t)(1-2t)B_{N,d}(t)-(1-2t)^{1-\beta}B_{N,d}(t)\right| =O_P(N^{\frac{\beta-1}{2}})=o_P(1). $$
Hence for $0<z<1/2$, 
$$
\frac{(Nd)^{1/2}}\sigma\bigg(\sup_{z\leq t \leq 1-z} \frac{|V_{N,d}(t)|}{w(t)},\sup_{z\leq t \leq 1-z} \frac{|Z_{N,d}(t)|}{w(t)} \bigg) \stackrel{\cD}{\to}\bigg(\sup_{z\leq t\leq  1-z}\frac{|B(t)|}{w(t)},\sup_{z\leq t\leq  1-z}\frac{|1-2t|^{1-\beta}|B(t)|}{w(t)}\bigg)
$$
where $B(t)$ is a Brownian bridge. Using again \eqref{pr1-2} we get
\begin{align}
&\sup_{2/N\leq t \leq z}\frac{1}{w(t)}\left|\tfrac{(Nd)^{1/2}}{\sigma}V_{N,d}(t)-B_{N,d}(t)\right|+ \sup_{2/N\leq t \leq z}\frac{1}{w(t)}\left|\tfrac{(Nd)^{1/2}}{\sigma}Z_{N,d}(t)-(1-2t)^{1-\beta}B_{N,d}(t)\right|\notag\\
&\leq \sup_{0<s\leq z}\frac{s^{1/2}}{w(s)}\bigg(\sup_{2/N\leq t \leq z}\frac{1}{t^{1/2}}\left|\tfrac{(Nd)^{1/2}}{\sigma}V_{N,d}(t)-B_{N,d}(t)\right|\notag\\
& \hspace{6.5cm} + \sup_{2/N\leq t \leq z}\frac{1}{t^{1/2}}\left|\tfrac{(Nd)^{1/2}}{\sigma}Z_{N,d}(t)-(1-2t)^{1-\beta}B_{N,d}(t)\right|\bigg)\notag\\
&\leq \sup_{0<s\leq z}\frac{s^{1/2}}{w(s)}\left(\frac{2}{N}\right)^{\bar{\alpha}-\tfrac  12}\bigg(\sup_{2/N\leq t \leq z}\frac{1}{t^{\bar{\alpha}}}\left|\tfrac{(Nd)^{1/2}}{\sigma}V_{N,d}(t)-B_{N,d}(t)\right|\notag\\
& \hspace{6.5cm}+\sup_{2/N\leq t \leq z}\frac{1}{t^{\bar{\alpha}}}\left|\tfrac{(Nd)^{1/2}}{\sigma}Z_{N,d}(t)-(1-2t)^{1-\beta}B_{N,d}(t)\right|\bigg)\notag\\
&=\sup_{0<s\leq z}\frac{s^{1/2}}{w(s)}\Big(O_P(1) + O_P(N^{\frac\beta2-\bar\alpha-\frac{1}{2}})\Big).\label{pr1-3}
\end{align}
It is shown in  \citet[p. 179]{csh-1} that $\lim_{z\to 0}t^{1/2}/w(t)=0.$
Moreover, according to  \citet[p. 179]{csh-1} 
$$
\sup_{0<t\leq z}\frac{|B(t)|}{w(t)}\stackrel{P}{\to}c_1,\quad z\to 0,
$$
for some $0\leq c_1<\infty$. Therefore, by \eqref{pr1-3}
$$
\lim_{z\to 0}\limsup_{N\to \infty}P\left\{  \left|\sup_{2/N\leq t \leq z}\frac{1}{w(t)}\left|\sigma^{-1}(Nd)^{1/2}V_{N,d}(t)\right|-c_1\right|>x  \right\}=0
$$
for all $x>0$. The same arguments give
$
\sup_{1-z\leq t<1}\frac{|B(t)|}{w(t)}\stackrel{P}{\to}c_2,$ as $z\to 0$
for some $c_2\geq 0$, and %
$$
\lim_{z\to 0}\limsup_{N\to \infty}P\left\{  \left|\sup_{1-z\leq t \leq 1-2/N}\frac{1}{w(t)}\left|\sigma^{-1}(Nd)^{1/2}V_{N,d}(t)\right|-c_2\right|>x  \right\}=0
$$
for all $x>0$; similar arguments apply to $Z_{N,d}$.  Since
$$
\sup_{0<t<1} \frac{\max\{ |B(t)|,|1-2t|^{1-\beta}|B(t)|\}}{w(t)} =\sup_{0<t<1}\frac{|B(t)|}{w(t)},
$$
the proof of \eqref{th1-1} is complete. %
\qed

\medskip
\noindent
{\bf Proof of Theorem \ref{th2}.} Following arguments leading to \eqref{pr1-2}, using Lemmas \ref{lem-1} and \ref{lem-2}, with $B_{N,d}(t)$ as in \eqref{e:B_Nd}, we obtain %
\begin{align}\label{da-pr1}
\sup_{2/N\leq t \leq 1-2/N}\frac{1}{(t(1-t))^{1/2}}
&\left|\frac{N^{1/2}}{2\cs(d)}Z_{N,d}(t)- (1-2t)^{1-\beta}B_{N,d}(t)\right| \notag\\ &+  
\sup_{2/N\leq t \leq 1-2/N}\frac{1}{(t(1-t))^{1/2}}
\left|\frac{N^{1/2}}{2\cs(d)}V_{N,d}(t)-B_{N,d}(t)\right|=O_P(1).
\end{align}

For $t$ in the range $2/N<t<1/2$, it follows from Cs\"org\H{o} and Horv\'ath (1993, p.\ 256)
\begin{align*}
&\frac{1}{(2\log \log N)^{1/2}}\max_{2/N<t<1/2}\frac{1}{t^{1/2}}\left| B_{N,d}(t)\right|\\
&=\frac{1}{(2\log \log N)^{1/2}}\max_{2<u\leq N/2}\frac{1}{u^{1/2}}\Big|   W_{N,d,1}(u)-\tfrac{u}{N}\big(  W_{N,d,1}(N/2)  + W_{N,d,2}(N/2)\big)    \Big|\stackrel{P}{\to}1,
\end{align*}
$$
\max_{1/N\leq t \leq (1/N)(\log N)^4}\frac{1}{t^{1/2}}\left|B_{N,d}(t)  \right|
=O_P\left(\left(\log\log\log N\right)^{1/2}\right).
$$
and 
$$
\max_{1/(\log N)^4\leq t\leq 1/2}\frac{1}{t^{1/2}}\left|B_{N,d}(t)  \right|
=O_P\left(\left(\log\log\log N\right)^{1/2}\right).
$$
For convenience write $S_{N,d}(t) =(2\cs(d))^{-1}N^{1/2}\max \{|V_{N,d}(t)|,|Z_{N,d}(t)|\}$. Since $$\max\{ |B_{N,d}(t)|,|1-2t|^{1-\beta}|B_{N,d}(t)|\}= |B_{N,d}(t)|,$$
 by \eqref{da-pr1} 
\begin{equation}\label{e:S_Nd_bound}
\sup_{0<t<1/2} |B_{N,d}(t)|-O_P(1) \leq \sup_{0<t<1/2}S_{N,d}(t)\leq \sup_{0<t<1/2} |B_{N,d}(t)|+O_P(1),
\end{equation}
 and we obtain %
\beq\label{da-pr2}
\frac{\sup_{0<t<1/2}S_{N,d}(t)}{(2\log \log N)^{1/2}}  \stackrel P \to 1,%
\eeq
\beq\label{da-pr2/3}
\max_{1/N\leq t \leq (1/N)(\log N)^4} S_{N,d}(t)=O_P\left(\left(\log\log\log N\right)^{1/2}\right),
\eeq
and
\beq\label{da-pr3}
\max_{1/(\log N)^4\leq t\leq 1/2} S_{N,d}(t)=O_P\left(\left(\log\log\log N\right)^{1/2}\right).
\eeq
Putting together  \eqref{da-pr2}--\eqref{da-pr3} we conclude
\begin{align*}
\lim_{N,d\to \infty}P\Biggl\{\sup_{0<t \leq 1/2}&S_{N,d}(t)=\sup_{N^{-1}(\log N)^4\leq t \leq 1/(\log N)^4}S_{N,d}(t)\Biggl\}=1.
\end{align*}
By \eqref{e:S_Nd_bound} we obtain 
\begin{align*}
\max_{(\log N)^4\leq k \leq N/(\log N)^4}\left|S_{N,d}(k/N)
-\frac{1}{k^{1/2}}W_{N,d,1}(k)\right|=o_P\left(\left(\log\log N\right)^{-1/2}\right).
\end{align*}
The classical Darling--Erd\H{o}s limit result yields (cf.\  Cs\"org\H{o} and Horv\'ath, 1993, p.\ 256)

\begin{align}\label{erdos1}
\lim_{N\to \infty}P\left\{a(\log N)\max_{(\log N)^4\leq k \leq N/(\log N)^4}\frac{1}{k^{1/2}}\left| W_{N,d,1}(k)  \right|\leq x +b(\log N)\right\}
=\exp(-e^{-x})
\end{align}
for all $x$, implying 
\begin{equation}\label{e:S(t)_near0}
P\left\{a(\log N)\sup_{0\leq t \leq 1/2}S_{N,d}(t) \leq x +b(\log N)\right\}
\to \exp(-e^{-x}).
\end{equation}
 Analogously, since
\begin{align*}
\max_{N-N/(\log N)^4\leq k \leq N-(\log N)^4}\left|S_{N,d}(k/N)
-\frac{1}{(N-k)^{1/2}}W_{N,d,2}(N-k)\right|=o_P\left(\left(\log\log N\right)^{-1/2}\right).
\end{align*} %
then \eqref{e:S(t)_near0} holds when the $\sup$ instead taken over $1/2< t \leq 1$, %
 and \eqref{th2-1} follows from %
  the independence of $W_{N,d,1}$ and $W_{N,d,2}$.  %

\qed

\medskip
\noindent
{\bf Proof of Theorems \ref{th3} and \ref{th4}.} In Lemma \ref{appesecmom}, the orders of $\cm(d,\nu)$ and $\cs(d)$ are established under Assumption \ref{as6}.  Combining this with Lemmas \ref{lem-1} and \ref{lem-2},  the proofs of Theorems \ref{th1} and \ref{th2} can be repeated with minor  modifications to establish the result.
\qed

\medskip

\noindent {\bf Proof of Propositions \ref{p:jack1} and \ref{p:jack2}.} 
Define, for $c=0,1,2$,
$$
U_c = \left( {N\choose c}  {N-c \choose 2-c}  {N-2\choose 2-c} \right)^{-1}\sum_c \psi_{i,j} \psi_{k,\ell}
$$
where $\sum_c$ denotes the sum over pairs $(i,j)$, $i\neq j$ and  $(k,\ell)$, $k\neq \ell$ in $\{1,\ldots,N\}$ with $c$ indices in common. %
Using $\frac{1}{N}\sum_{i=1}^N U_{N-1,d}^{(-i)}= U_{N,d}$, we can reexpress (\citet{arvesen:1969}, p. 2081)
\begin{align}
\hat{\sigma}_{N,d}^2 & = (N-1) \sum_{i=1}^N (U_{N-1,d}^{(-i)}- U_{N,d})^2\notag\\
& = \frac{(N-1)}{N}{ {N-1\choose 2}^{-2}} \sum_{c=0}^2(cN-4){N\choose c}{N-c \choose 2-c} {N-2\choose 2-c} U_c\notag\\
& =  O(1)U_0 + (4+o(1))U_1 + O(N^{-1}) U_2. \label{e:jackknife_decomp_Uc}
\end{align}
Applying Theorem 1.3.3 in \cite{lee:1990}, we obtain $\Var (U_c) \leq O(N^{-1} \E \psi_{1,2} ^4)$, $c=0,1,2$, so under the conditions of Theorem \ref{th1} or of Theorem \ref{th3},
\beq\label{e:Uc_estimates}
U_c = \E U_c + O_P(N^{-1/2} (\E \psi_{1,2} ^4)^{1/2}),\quad c=0,1,2.
\eeq   %
Note $\E U_0 = 0$, $\E U_1 = \E \psi_{1,2}\psi_{1,3} = \E\big[\E_{\bf X_1} \psi_{1,2} 
E_{\mathbf X_1}\psi_{1,3}\big] = E \zeta_1^2 = \cs^2(d),$ and $\E U_2 = E \psi_{1,2}^2$.  Thus, under the conditions of Theorem \ref{th1}, Lemma \ref{appe-mom} shows $\E \psi_{1,2} ^4=O(d^{-2})$, and using \eqref{e:jackknife_decomp_Uc}, we get%
$$%
 \hat \sigma^2_{N,d} = O_P(d^{-1}N^{-1/2}) + (4+o(1))\big[\cs^2(d) + O_P(N^{-1/2}d^{-1})\big] + O_P( N^{-1}d^{-1})   %
$$%
which, by \eqref{th2-3}, implies $ d\sigma^2_{N,d}  = 4d \cs^2(d) + o_P(1) \stackrel P \to \sigma^2$. 

On the other hand, under the conditions assumed in Theorem \ref{th3} we have $\E \psi_{1,2} ^4=O(1)$, so by \eqref{e:jackknife_decomp_Uc} and \eqref{e:Uc_estimates},%
\beq\label{ja-2}
\hat{\sigma}_{N,d}^2\stackrel{P}{\to}\gamma^2, \quad\mbox{as}\;\;\min\{N,d\}\to \infty.
\eeq
The proof of {Proposition} \ref{p:jack2} is similar and is omitted.

\qed

\noindent \textbf{Proof of Proposition \ref{th:consistency}}.  Define
\begin{equation}\label{e:def_qV}
q_V(t)  = \begin{cases}
\left(\displaystyle\frac{1-\eta}{1-t}\right)^2\big(\mu_1-\mu_2\big) - 2\left(\displaystyle\frac{1-\eta}{1-t}\right)\left(1-\displaystyle\frac{1-\eta}{1-t}\right)\big(\mu_{12}-\mu_1\big) & t\leq \eta,\vspace{0.5ex}\\
\left(\displaystyle\frac{\eta}{t}\right)^2\big(\mu_1-\mu_2\big) - 2\left(\displaystyle\frac{\eta}{t}\right)\left(1-\displaystyle\frac{\eta}{t}\right)\big(\mu_{12}-\mu_2\big) & t>\eta,
\end{cases}
\end{equation}
and
\begin{equation}\label{e:def_qZ}
q_Z(t) = \bigg[\eta^2 - {\mathbbm 1}_{\{t\leq \eta\}}\left(1-\displaystyle\frac{1-\eta}{1-t}\right)\bigg]\big(\mu_{1,2}-  \mu_1 \big) + \bigg[(1-\eta)^2 - {\mathbbm 1}_{\{t> \eta\}}\left(1-\displaystyle\frac{\eta}{t}\right)\bigg] \big(\mu_{1,2}-\mu_2\big)
\end{equation}
 We first consider $(ii)$. %
For $0\leq t\leq \eta$, with $u=\frac{1-\eta}{1-t} \in[1-\eta,1]$, observe
$$
\frac{|q_Z(t)|}{|\mu_1-\mu_2|} =  \big|u^2 - 2u(1-u)\times o(1) \big|
$$
is eventually increasing  for all $1-\eta\leq u\leq1$,  i.e., for all $0\leq t\leq \eta$, so $\argmax_{0\leq t \leq \eta}|q_Z(t)|=\eta$ for all large $N,d$. Arguing analogously for $\eta \leq t \leq 1$, we therefore obtain $\argmax_{0\leq t \leq 1}|q_Z(t)|=\eta$ for all large $N,d$, i.e., $\eta$ is the unique maximizer of $q_Z(t)$ for all large $N,d$.  Now, applying Lemma \ref{l:onechange_lemma},
\beq\label{e:z-q_z}
\frac{|q_V(t)|}{|\mu_1-\mu_2|} - o_P(1) \leq \frac{|V_{N,d}(t)|}{d^{-1/p}|\mu_1-\mu_2|} \leq \frac{|q_V(t)|}{|\mu_1-\mu_2|} + o_P(1),
\eeq
for all $2/N<t<1-2/N$.
This gives
$$
\frac{|q_V(\eta)|}{|\mu_1-\mu_2|} - o_P(1) \leq  \frac{|V_{N,d}(\eta)|}{d^{-1/p}} \leq  \frac{|V_{N,d}(t_V)|}{d^{-1/p}} \leq  \frac{|q_V(t_V)|}{|\mu_1-\mu_2|} + o_P(1).
$$
Thus, for every small $\delta>0$,
\beq\label{e:tV_to_t0}
\P\Big( |t_V-\eta|>\delta \Big) \leq P\bigg(\frac{|q_V(\eta)|}{|\mu_1-\mu_2|} - o_P(1) \leq  \frac{\sup_{t:|t-\eta|\geq \delta}|q_V(t)|}{|\mu_1-\mu_2|} + o_P(1)\bigg) \to 0.
\eeq
Further, since $M_{N,d}=\max\{|\mu_{1,2}-\mu_1|,|\mu_{1,2}-\mu_2|\}\gg d^{\frac{1}{p}-\frac{1}{2}}N^{-1/2}$ applying Lemma \ref{l:onechange_lemma} again we have $|Z_{N,d,0}(t)|/M_{N,d} = o_P(1) +|q_Z(t)|/M_{N,d}$ for all $2/N<t<1-2/N$.  It is easily seen $\argmax_t|q_Z(t)|\subseteq\{0,\eta,1\}$ for all large $N,d$ (its exact values depending on the magnitude and signs of $\mu_{1,2}-  \mu_1$ and $\mu_{1,2}-  \mu_2$), and arguing in a similar fashion as for \eqref{e:tV_to_t0}, for each $N,d$$$
\P\Big( |t_Z|>\delta,~|t_Z-\eta|>\delta,~|t_Z-1|>\delta\Big) \leq P\bigg(\frac{|q_Z(t_Z^*)|}{M_{N,d}} - o_P(1) \leq  \frac{\sup_{t:\max\{|t-\eta|,t,1-t\}\geq \delta}|q_Z(t)|}{M_{N,d}} + o_P(1)\bigg) \to 0.
$$
Thus, may pick a sequence $t^*_Z\in \argmax_t|q_Z(t)|$ such that
\begin{equation}\label{e:argmax_Z}
|t^*_Z-t_Z|=o_P(1).
\end{equation} 
  Further, clearly, by condition $(ii)$,
$$
\frac{\sup_{0<t<1}|Z_{N,d,0}(t)|}{d^{-1/p}|\mu_1-\mu_2|} \leq \frac{\sup_{0<t<1}|q_{Z}(t)|}{d^{-1/p}|\mu_1-\mu_2|} +\frac{\sup_{0<t<1}|Z_{N,d,0}(t)-q_{Z}(t)|}{d^{-1/p}|\mu_1-\mu_2|} = o_P(1),
$$
whereas by \eqref{e:z-q_z}, and the definition of $q_V$,  $\sup_{0<t<1}|V_{N,d}(t)|/(d^{-1/p}|\mu_1-\mu_2|) \stackrel \P \to  1,$ showing $ |Z_{N,d,0}(t_Z)| = o_P(|V_{N,d,0}(t_V)|).$  In turn, this implies, in the case that $\eta\neq 1/2$,
\begin{equation}\label{e:Z_to_Z_0}
 |Z_{N,d}(t_Z)| =O_P(1) |Z_{N,d,0}(t_Z)|=o_P(|V_{N,d}(t_V)|).
\end{equation}
This gives $\widehat \eta_{N,d}=t_V$ with probability tending to 1, which together with \eqref{e:tV_to_t0} gives $\widehat \eta_{N,d}\stackrel P \to \eta$.  If $\eta=1/2$,  on the event $\{\omega: t_Z^*\in\{0,1\}\}$, \eqref{e:Z_to_Z_0} still holds, again implying $\widehat \eta_{N,d}\stackrel P \to \eta$,  and clearly on the event $\{\omega:t_Z^*=1/2\}$,  we have $t_Z\stackrel P \to 1/2$. Since we also have $t_V\stackrel P\to \eta=1/2$ then $\widehat \eta_{N,d}\stackrel P \to 1/2 $ and the result is proven under condition $(ii)$.  %

	 For condition $(i)$, note $\mu_{1,2}-\min\{\mu_1,\mu_2\}\geq \mu_{1,2}-\max\{\mu_1,\mu_2\}\gg \max\{\mu_1,\mu_2\}-\min\{\mu_1,\mu_2\}=|\mu_1-\mu_2|$. From this it is readily seen under $(i)$ that $\argmax |q_Z(t)|=\eta$ for all large $N,d$, and arguing similarly as in the case of $V_{N,d}$, we have
\begin{equation}\label{e:argmax_Z}
t_Z \stackrel \P\to  \eta.
\end{equation}
Now, observe, for $0 \leq t \leq \eta$,
$$
\frac{|q_V(t)|}{\mu_{1,2} }  = 2\left(\frac{1-\eta}{1-t}\right)\left(1-\frac{1-\eta}{1-t}\right) + o(1) \leq \frac{1}{2} + o(1).  %
$$
Arguing analogously for $\eta \leq t \leq 1$, we have $|q_V(t)| /\mu_{1,2}  \leq \frac{1}{2}+ o(1)$, giving
 $\limsup_{N,d\to\infty} \sup_{0\leq t \leq 1}|q_V(t)|/\mu_{1,2} \leq \frac{1}{2},$ which implies
 $$
 \frac{|V_{N,d}(t_Z)|}{d^{-1/p}\mu_{1,2}} = o_P(1) + \frac{\sup_{0\leq t \leq 1}|q_V(t)|}{\mu_{1,2} } \leq \frac{1}{2} + o_P(1).
 $$
  On the other hand, by \eqref{e:argmax_Z}, %
 $$
\frac{\sup_{0<t\leq1}|Z_{N,d,0}(t)|}{d^{-1/p}\mu_{1,2}} = o_P(1) + \frac{\sup_{0<t\leq1}|q_Z(t)|}{\mu_{1,2}} \stackrel P \to \eta^2 +(1-\eta)^2\geq \frac{1}{2}
 $$
so $$\frac{|Z_{N,d}(t_Z)|}{\mu_{1,2}}\stackrel P \to \frac{\eta^2 +(1-\eta)^2}{|1-2\eta|^{\beta}} > \frac{1}{2} + \delta,$$ for some $\delta>0$.  In other words,
\begin{align*}\P(|Z_{N,d}(t_Z)|>|V_{N,d}(t_V)|) &=  \P\Bigg(\frac{|Z_{N,d}(t_Z)|}{\mu_{1,2}} \geq \frac{|V_{N,d}(t_Z)|}{\mu_{1,2}}\Bigg)\\
& \geq \P\Bigg(\frac{1}{2}+\delta +o_P(1) \geq \frac{1}{2}+o_P(1)\Bigg) \to 1,
\end{align*}
which gives the result under condition $(i)$. Cases $(i'),(ii')$ can be established arguing analogously to $(i),(ii)$ with minor changes.\qed\\

\noindent \textbf{Proof of Theorem \ref{th:power}.} We provide proofs for $(i)$ and $(ii)$ under Assumptions \ref{as4}--\ref{as5}, since cases $(i'),(i'')$ follow from a similar argument.  First we consider $(i)$ in the case that $(Nd)^{1/2}d^{-1/p}|\mu_1-\mu_2|\to \infty$. Observe that
\begin{align*}
d^{-1/p}U_{N,d,1}(k_1)&= d^{-1/p}\mu_1 +   {k\choose 2}^{-1}\sum_{1\leq i<j\leq k}\psi_{i,j}\\
& = d^{-1/p}\mu_1+ 2\frac{1}{k} \sum_{i=1}^k\zeta_i  + {k\choose 2}^{-1}\sum_{1\leq i<j\leq k}y_{ij}\\
& = d^{-1/p}\mu_1+O_P( N^{-1/2} \cs(d)) + O_P\big(N^{-1}\left(E\psi^2_{1,2}+E\zeta_1^2  \right)^{1/2}\big)\\
& = d^{-1/p}\mu_1+O_P((Nd)^{-1/2}),
\end{align*}
where we used Lemma \eqref{appe-mom} on the last line. Analogously, $U_{N,d,2}(k_1)=\mu_2+O_P\left(d^{1/p-1/2}/N^{1/2}\right).$ Thus, for some $c_0,c_1>0$,
\begin{align}T_{N,d}\geq c_0 |V_{N,d}(k_1/N)|&\geq c_1 d^{-1/p}|U_{N,d,2}(k_1)-U_{N,d,2}(k_1)|\notag\\
& \geq c_1 d^{-1/p}|\mu_1-\mu_2| +  O_P\big((Nd)^{-1/2}\big).\label{e:TNd_lowerbound}
\end{align}
Thus, for $\widetilde T_{N,d}$ as in \eqref{e:def_rescaledT_Nd}, we obtain $\widetilde T_{N,d} \geq c_1 (Nd)^{1/2}d^{-1/p} |\mu_1-\mu_2| +  O_P(1) \stackrel P \to \infty$, giving $(i)$ whenever $(Nd)^{1/2}d^{-1/p}|\mu_1-\mu_2|\to \infty$.  If $(Nd)^{1/2}d^{-1/p}|\mu_{1,2}-\mu_2|\to \infty$ but $(Nd)^{1/2}d^{-1/p}|\mu_1-\mu_2| = O(1)$, applying Lemma \ref{l:onechange_lemma}, we obtain, for some $\delta>0$,
\begin{align*}(Nd)^{1/2}T_{N,d}&\geq  C(Nd)^{1/2} |V_{N,d}(\eta-\delta)|\\
&\geq C(Nd)^{1/2}d^{-1/p}   |q_Z(\eta-\delta)| + O_P(1) \\
&\geq C (Nd)^{1/2}d^{-1/p}|\mu_{1,2}-\mu_1| + O_P(1),
\end{align*}
showing $\widetilde T_{N,d}\stackrel P\to \infty$. An analogous argument shows $\widetilde T_{N,d}\stackrel P\to \infty$ when $(Nd)^{1/2}d^{-1/p}|\mu_{1,2}-\mu_2|\to \infty$ but $(Nd)^{1/2}d^{-1/p}|\mu_1-\mu_2| = O(1)$, giving \eqref{e:widetildeT_diverg} in case $(i)$.

For $(ii)$, arguing as in \eqref{e:TNd_lowerbound} we have
\begin{align*}
a(\log N) (Nd)^{1/2}T_{N,d} &\geq C  [\log \log N]^{1/2}N^{1/2}d^{\frac{1}{2}-\frac{1}{p}}|\mu_1-\mu_2| +  O_P\big(\log\log (N)\big)\\
&= C  \log\log N\Big([\log \log N]^{-1/2}N^{1/2}d^{\frac{1}{2}-\frac{1}{p}}|\mu_1-\mu_2| +  O_P(1)\Big),%
\end{align*}
which, since $(Nd)^{1/2} \left(\log\log N\right)^{-1/2}d^{-1/p}|\mu_1-\mu_2| \to \infty$, implies $\widetilde T_{N,d}\stackrel P \to \infty$.  If $(Nd)^{1/2} \left(\log\log N\right)^{-1/2}d^{-1/p}|\mu_1-\mu_2|-2\log \log N  =O(1)$ but $(Nd)^{1/2} \left(\log\log N\right)^{1/2}d^{-1/p}|\mu_{1,2}-\mu_i| \to\infty,$ we may apply Lemma \ref{l:onechange_lemma} as in the case of $(i)$ to again obtain $\widetilde T_{N,d}\stackrel P\to\infty$, giving \eqref{e:widetildeT_diverg} in case $(ii)$.  \qed \medskip

\medskip
\section{Auxiliary lemmas}\label{app-mom}

\begin{lemma}\label{l:mensh}
Expression \eqref{e:maximal_inequ} holds. 
\end{lemma}
\begin{proof}
Note the variables $y_{i,j}=\psi_{i,j}-\zeta_i-\zeta_j$ are uncorrelated for every pair $(i,j)\neq(i',j')$. For each fixed integer $k\in[m_1,m_2]$,  define
$$\widetilde y(i,j) = y_{i,N-j+1}I\{{i\leq m_2,j\leq N-m_1\}}.$$
By considering the binary expansions $k=\sum_{i=0} ^{\nu_1}a_i2^{\nu_1-i}$ and $N-k=\sum_{j=0}^{\nu_2} b_j2^{\nu_2-j}$   ($a_i,b_j\in\{0,1\}$), we may rewrite
$$
\sum_{i=1}^k\sum_{j=k+1}^N y_{ij}=\sum_{i=1}^{k}\sum_{j=1}^{N-k}y_{i,N-j+1}= \sum_{i=0}^{\nu_1} \sum_{j=0}^{\nu_2} \eta_{ij}
$$
where $ \eta_{ij}=0$ when $a_ib_j=0$, and for $a_i=b_j=1$,
$$
\eta_{i,j} = \sum_{s=1}^{2^{(\nu_1-i)}} \sum_{t=1}^{2^{(\nu_2-j)}} \widetilde y(A_i+s,B_j+t)
$$
with $A_i= \sum_{\ell=0}^{i-1} a_\ell 2^{\nu_1-\ell}$ and $B_j = \sum_{\ell=0}^{j-1} b_\ell 2^{\nu_2-\ell}$. %
Thus, by Cauchy-Schwarz,
$$
\bigg(\sum_{i=1}^k\sum_{j=k+1}^N y_{ij} \bigg)^2\leq (\nu_1+1)(\nu_2+1)  \sum_{i=0}^{\nu_1} \sum_{j=0}^{\nu_2} \eta_{ij}^2
$$
We now provide an a.s. uniform bound for $\eta_{i,j}^2$.  Let $r_1,r_2$ be the unique integers such that
$$
 2^{r_1}<N-m_1\leq 2^{r_1+1}=L,\qquad 2^{r_2}< m_2\leq 2^{r_2+1}=M.
$$
And define
$$
S_{i,j} = \sum_{\ell=0}^{M/2^i-1} \sum_{\ell'=0}^{L/{2^j}-1}\bigg(\sum_{s=1}^{2^i} \sum_{t=1}^{2^j} \widetilde y(s+\ell 2^i,t+\ell' 2^j)  \bigg)^2
$$
Since $\widetilde y(\ell,\ell')=0$ for $\ell>m_2$ and $\ell'>N-m_2$, we have, by orthogonality of $\widetilde y(\ell,\ell')$,
\begin{align*}
\E S_{i,j} &=\sum_{\ell=0}^{M/2^i-1} \sum_{\ell'=0}^{L/{2^j}-1} \sum_{s=1}^{2^i} \sum_{t=1}^{2^j} E\widetilde y(s+\ell 2^i,t+\ell' 2^j)^2\\
&=  \sum_{\ell=1}^{m_2}\sum_{\ell=1}^{N-m_1}\E \widetilde y(\ell,\ell')^2 = m_2(N-m_1) \E y_{1,2}^2.
\end{align*}
So that $\E  \sum_{\ell=0}^{r_1+2}\sum_{\ell'=0}^{r_2+2}S_{\ell,\ell'} \leq (r_1+1)(r_2+1) m_2(N-m_1) \E y_{1,2}^2$. Notice for each pair $(i,j)$ with $1\leq i \leq m_2$ and $1\leq j \leq N-m_1$, $\eta^2_{i,j}$ appears as a summand in $S_{\nu_1-i,\nu_2-j}$, implying $\eta_{i,j}^2 \leq \sum_{\ell=0}^{r_1+1}\sum_{\ell'=0}^{r_2+1}S_{\ell,\ell'}$ for every such $i,j$. Putting this together with the above, we obtain
$$
E\max_{m_1\leq k \leq m_2}\left(\sum_{i=1}^k\sum_{j=k+1}^N y_{ij}\right)^2 \leq (r_1+2)^2(r_2+2)^2m_2(N-m_1) \E y_{1,2}^2,
$$
which gives the desired bound.
\end{proof}

\begin{lemma}\label{l:onechange_lemma}
Suppose the hypotheses of Theorem \ref{th:power} hold. Let $\eta=k_1/N.$ Then,
$$
\sup_{2/N<t<1-2/N}\ \big| V_{N,d}(t) - d^{-1/p}q_V(t)\big|= O_P((N\cs(d)^2)^{-1/2})
$$
and
$$
\sup_{2/N<t<1-2/N} \big| Z_{N,d,0}(t) - d^{-1/p}q_Z(t)\big| = O_P((N\cs(d)^2)^{-1/2}),
$$
where $q_V(t)$ and $q_Z(t)$ are defined as in \eqref{e:def_qV} and \eqref{e:def_qZ}.
\end{lemma}
\begin{proof}
We first turn to $V_{N,d}$.  A straightforward computation shows $\E[V_{N,d}(t)] = d^{-1/p}q_{V}(t)$, and we can reexpress
\begin{align*}
(Nd)^{1/2}&\big(V_{N,d}(t) - d^{-1/p}q_{V}(t)\big) \\
&= N^{1/2}\frac{k}{N}\frac{k^*}{N}  2 d^{1/2} \bigg(\frac{1}{k(k-1)} \sum_{1\leq i<j\leq k}\psi_{i,j} -\frac{1}{k^* (k^*-1)} \sum_{k+1\leq i<j\leq N}\psi_{i,j}\bigg)%
\end{align*}
Hence, using arguments along the lines of those in Theorem \eqref{th1}, we obtain $$(N\cs^2(d))^{1/2}\sup_{2/N \leq t \leq 1-2/N}|V_{N,d}(t) - d^{-1/p}q_{V}(t)\big|=O_P(1).$$ A similar reasoning applies to $Z_{N,d}$, establishing the claim.

\end{proof}

The proofs in Sections \ref{sec-pr1} and \ref{sec-pr2} use the following lemma when the coordinates satisfy the weak dependence assumption.%
\begin{lemma}\label{appe-mom} If Assumptions \ref{as4}--\ref{as5} hold, then, %
\beq\label{ori}
E\psi_{1,2}^4\leq \frac{c}{d^2},
\eeq
\beq\label{ori2}
{\cs^2(d)=}E\zeta_1^2=\frac{\sigma^2}{4d}(1+o(1)) \quad{ d\to \infty.}
\eeq
Also, %
\beq\label{ori3}
E|\zeta_1|^4\leq \frac{c}{d^{4}},
\eeq
with $\sigma^2$ in \eqref{sidef},  $\psi_{1,2}$ and $\zeta_1$ are defined in \eqref{psize}. %
\end{lemma}
\begin{proof} We first establish \eqref{ori}. Let $\bZ=\bX_1-\bX_2=(Z_1, Z_2, \ldots, Z_d)^\T$, so 
\begin{align*}
\psi_{1,2}&=  \Bigg[\bigg( \frac{1}{d}\sum_{j=1}^d |Z_j|^p  \bigg)^{1/p}-\bigg( \frac{1}{d}\sum_{j=1}^d E|Z_j|^p  \bigg)^{1/p}\Bigg]+ \Bigg[\bigg( \frac{1}{d}\sum_{j=1}^d E|Z_j|^p  \bigg)^{1/p}-\E \bigg( \frac{1}{d}\sum_{j=1}^d |Z_j|^p  \bigg)^{1/p}\bigg]\\
& = T_1+T_2.
\end{align*}
 By  the mean value theorem,
\begin{align}\label{meanv}
T_1=\bigg(\frac{1}{d}\sum_{j=1}^d |Z_j|^p  \bigg)^{1/p}-\bigg( \frac{1}{d}\sum_{j=1}^d E|Z_j|^p  \bigg)^{1/p}
=\frac{1}{p}\xi^{-1+1/p}\frac{1}{d}\sum_{j=1}^d[|Z_j|^p- E|Z_j|^p ],
\end{align}
where $\xi$ lies between $(1/d)\sum_{j=1}^d |Z_j|^p$ and $(1/d)\sum_{j=1}^d  E |Z_j|^p$, i.e.,
\beq\label{ta1}
\bigg| \xi- \frac{1}{d}\sum_{j=1}^d E|Z_j|^p   \bigg|\leq \bigg|\frac{1}{d}\sum_{j=1}^d[|Z_j|^p- E|Z_j|^p ]\bigg|.
\eeq
By Markov's inequality we have for all $\alpha>0$ as in Assumption \ref{as4},
\begin{align*}
P\Bigg\{\bigg| \xi- \frac{1}{d}\sum_{j=1}^d E|Z_j|^p   \bigg|>\frac{1}{2d}\sum_{j=1}^dE|Z_j|^p\Bigg\}%
&\leq P\Bigg\{\bigg|\frac{1}{d}\sum_{j=1}^d[|Z_j|^p-E|Z_j|^p]\bigg|>\frac{1}{2d}\sum_{j=1}^dE|Z_j|^p\Bigg\}\\
&\leq \bigg(\frac{1}{2d}\sum_{j=1}^dE|Z_j|^p\bigg)^{-\alpha}E\Bigg|\frac1d\sum_{j=1}^d[|Z_j|^p-E|Z_j|^p]\Bigg|^\alpha\\
&\leq c_1d^{-\alpha/2}
\end{align*}
by Assumptions \ref{as4} and \ref{as5} with some $c_1$. Since $\frac{1}{d}\sum_{j=1}^dE|Z_j|^p \to \ca(p)>0$ there are  constants $c_2>0$ and $c_3>0$ such that
 \beq\label{Adef}
P\{A\}\leq c_2d^{-\alpha/2},
\quad
\mbox{where}
\quad
A=\left\{\omega:\;  \xi\leq c_3   \right\}.
\eeq
Since $p\geq 1$ we get (cf.\ Hardy et al.,\ 1934, p.\ 44)
$$
\bigg(\sum_{j=1}^d|Z_j|^p\bigg)^{1/p}\leq \sum_{j=1}^d|Z_j|\quad\;\mbox{and}\quad\;\bigg(\sum_{j=1}^dE|Z_j|^p\bigg)^{1/p}\leq \sum_{j=1}^d(E|Z_j|^p)^{1/p},
$$
and therefore H\"older's inequality with $r= \alpha/{4}$, $r' =r/(r-1)=\alpha/(\alpha-4)$ yields
\begin{align}
E\bigg(T_1 I\{A\}\bigg)^4 &\leq\frac{1}{d^{4/p}} E\Bigg[\bigg(\sum_{j=1}^d|Z_j|+\sum_{j=1}^d(E|Z_j|^p)^{1/p}\bigg)^4I\{A\}\Bigg]\notag\\
&\leq \frac{1}{d^{4/p}}\Bigg\{E\bigg(\sum_{j=1}^d|Z_j|+\sum_{j=1}^d(E|Z_j|^p)^{1/p}\bigg)^{4r}\Bigg\}^{1/r}(P\{A\})^{1/r'}\notag\\
&\leq c_4 d^{4-(4/p)-(\alpha-4)/2}\notag\\
& = o(d^{-1}),\label{e:o(d^-1)_1}
\end{align}
since $6-(4/p)-\alpha/2<-1$ by Assumption \eqref{as2}.
By \eqref{meanv} we obtain
\begin{align}
E\Bigg[\Bigg(\bigg( \frac{1}{d}\sum_{j=1}^d |Z_j|^p  \bigg)^{1/p}-\bigg( \frac{1}{d}\sum_{j=1}^d E|Z_j|^p  \bigg)^{1/p}\Bigg)^2I\{\bar{A}\}\Bigg]
&\leq c_5 E\left( \frac{1}{d}\sum_{j=1}^d[|Z_j|^p- E|Z_j|^p ]   \right)^4\notag\\
&\leq \frac{c_6}{d^2}.\label{e:leq_cd^-1_1}
\end{align}
Thus, $$\E\Bigg[\bigg( \frac{1}{d}\sum_{j=1}^d |Z_j|^p  \bigg)^{1/p} - \bigg( \frac{1}{d}\sum_{j=1}^d E|Z_j|^p  \bigg)^{1/p}\Bigg]^2 \leq c d^{-1}.$$
  Next we show
\begin{align}
(T_2)^2=\Bigg( E \bigg(\frac{1}{d}\sum_{j=1}^d|Z_j|^p   \bigg)^{1/p}  - \bigg(\frac{1}{d}\sum_{j=1}^dE|Z_j|^p\bigg)^{1/p}\Bigg)^2
\leq c_7 \left(  d^{-1-\alpha/2} +d^{2-(2/p)-(\alpha/2)}\right)\label{e:E(z^p)-(Ez)^p}
\end{align}
which together with \eqref{e:o(d^-1)_1} and \eqref{e:leq_cd^-1_1} will imply \eqref{ori} since $\max\{-1-\alpha/2,~\! 2- 2/p-\alpha/2\}<-1$ giving $(T_2)^4=O(d^{-2}).$ So, a Taylor expansion gives
\begin{align}\label{meanv2}
&\hspace{-1cm}\bigg( \frac{1}{d}\sum_{j=1}^d |Z_j|^p  \bigg)^{1/p}-\bigg( \frac{1}{d}\sum_{j=1}^d E|Z_j|^p  \bigg)^{1/p}\\
&=\frac{1}{p}\bigg(  \frac{1}{d}\sum_{j=1}^d E|Z_j|^p   \bigg)^{-1+1/p}\frac{1}{d}\sum_{j=1}^d[|Z_j|^p- E|Z_j|^p ]\notag\\
&\hspace{1cm}+\frac{1}{p}\bigg(-1+\frac{1}{p}\bigg)\xi^{-2+1/p}\bigg(\frac{1}{d}\sum_{j=1}^d[|Z_j|^p- E|Z_j|^p ]\bigg)^2,\notag
\end{align}
where $\xi$ satisfies \eqref{ta1}. Similarly to our previous arguments, with $A$ as in \eqref{Adef}, we have
\begin{align*}
E\bigg(\xi^{-2+1/p}\bigg(\frac{1}{d}\sum_{j=1}^d[|Z_j|^p- E|Z_j|^p]\bigg)^4 I\{\bar{A}\}\bigg)\leq \frac{c_8}{d^4}E\bigg(\sum_{j=1}^d[|Z_j|^p- E|Z_j|^p]\bigg)^4
\leq \frac{c_{9}}{d^2}.
\end{align*}
By the Cauchy--Schwarz inequality and \eqref{Adef} we conclude
\begin{align*}
\bigg|E\frac{1}{d}\sum_{j=1}^d[|Z_j|^p- E|Z_j|^p ]I\{\bar{A}\}\bigg|
&\leq \bigg|E\frac{1}{d}\sum_{j=1}^d[|Z_j|^p- E|Z_j|^p ]I\{{A}\}\bigg|\\
&\leq \Bigg[E\bigg(\frac{1}{d}\sum_{j=1}^d[|Z_j|^p- E|Z_j|^p ]\bigg)^2\Bigg]^{1/2}(P\{{A}\})^{1/2}\\
&\leq c_{10}d^{-{1/2}-\alpha/4}.
\end{align*}
Using again \eqref{Adef} we obtain via the Cauchy--Schwarz inequality

\begin{align*}
&\hspace{-.5cm}\Bigg|E\Bigg( \bigg[\bigg(   \frac{1}{d}\sum_{j=1}^d |Z_j|^p  \bigg)^{1/p}-\bigg( \frac{1}{d}\sum_{j=1}^d E|Z_j|^p  \bigg)^{1/p} \bigg] I\{A\}\Bigg)\Bigg|\\
&\leq \frac{1}{d^{1/p}}\Bigg[E\bigg(\sum_{j=1}^d|Z_j|^p+\sum_{j=1}^d\left(E|Z_j|^p\right)^{1/p}\bigg)^2\Bigg]^{1/2}\left(P\{A\}\right)^{1/2}\\
&\leq c_{11}d^{1-(1/p)-(\alpha/4)}.
\end{align*}
This establishes \eqref{e:E(z^p)-(Ez)^p}, and thus \eqref{ori} is proven.\\

We now turn to \eqref{ori2}. Let $E_{\bX_1}$ denote the conditional expected value, conditioning with respect to $\bX_1$.
We write
\begin{align}
\zeta_1&=E_{\bX_1}\bigg( \frac{1}{d}\sum_{j=1}^d|X_{1,j}-X_{2,j}|^p  \bigg)^{1/p}-E\bigg( \frac{1}{d}\sum_{j=1}^d|X_{1,j}-X_{2,j}|^p  \bigg)^{1/p}\notag\\
&=E_{\bX_1}\bigg( \frac{1}{d}\sum_{j=1}^d|X_{1,j}-X_{2,j}|^p  \bigg)^{1/p}-\bigg( \frac{1}{d}\sum_{j=1}^dE_{\bX_1}|X_{1,j}-X_{2,j}|^p \bigg)^{1/p}
\label{e:zeta_3term_expansion}\\
&\hspace{1cm}+\bigg( \frac{1}{d}\sum_{j=1}^dE_{\bX_1}|X_{1,j}-X_{2,j}|^p  \bigg)^{1/p}
-\bigg( \frac{1}{d}\sum_{j=1}^dE|X_{1,j}-X_{2,j}|^p  \bigg)^{1/p}\notag\\
&\hspace{1cm}+\bigg( \frac{1}{d}\sum_{j=1}^dE|X_{1,j}-X_{2,j}|^p  \bigg)^{1/p}-E\bigg( \frac{1}{d}\sum_{j=1}^d|X_{1,j}-X_{2,j}|^p  \bigg)^{1/p}.\notag
\end{align}
Recall $g_j$ as defined in \eqref{def:gj}, and note $g_j(X_{1,j})=E_{\bX_1}|X_{1,j}-X_{2,j}|^p$. For the first term in \eqref{e:zeta_3term_expansion}, using a two term Taylor expansion, we get
\begin{align}\label{tay0}
&\hspace{-.5cm}E_{\bX_1}\bigg(\frac{1}{d}\sum_{j=1}^d|X_{1,j}-X_{2,j}|^p\bigg)^{1/p}-\bigg(\frac{1}{d}\sum_{j=1}^dg_j(X_{1,j})    \bigg)^{1/p}\\
&=\frac{1}{p}\bigg(\frac{1}{d}\sum_{j=1}^dg_j(X_{1,j})    \bigg)^{-1+1/p}\frac{1}{d}\sum_{j=1}^d[E_{\bX_1}|X_{1,j}-X_{2,j}|^p-g_j(X_{1,j})]\notag\\
&\hspace{1cm}+\frac{1}{p}\left(-1+\frac{1}{p}\right)E_{\bX_1}\bigg[\xi^{-2+1/p}\bigg(\frac{1}{d}\sum_{j=1}^d[|X_{1,j}-X_{2,j}|^p-g_j(X_{1,j})]\bigg)^2\bigg]
\notag\\
&=\frac{1}{p}\left(-1+\frac{1}{p}\right)E_{\bX_1}\bigg[\xi^{-2+1/p}\bigg(\frac{1}{d}\sum_{j=1}^d[|X_{1,j}-X_{2,j}|^p-g_j(X_{1,j})]\bigg)^2\bigg],\notag
\end{align}
where the random variable  $\xi$ now satisfies
\begin{align}\label{tay1}
\bigg|\xi-\frac{1}{d}\sum_{j=1}^dg_j(X_{1,j})\bigg|\leq
\bigg|\frac{1}{d}\sum_{j=1}^d[|X_{1,j}-X_{2,j}|^p-g_j(X_{1,j})]\bigg|.
\end{align}
We define again, as in \eqref{Adef},
$$
A=\{\omega:\; \xi\leq c_{12}\}
$$ but now $\xi$ is from \eqref{tay0} and \eqref{tay1}. Following the proof of \eqref{Adef}, we can choose $c_{12}$ such that with some $c_{13}$
\beq\label{tay2}
P\{A\}\leq c_{13}d^{-\alpha/2}.
\eeq
Squaring \eqref{tay0} and the using   Minkowski's inequality we have
\begin{align}\label{tay3}
&\Bigg( E_{\bX_1}\bigg( \frac{1}{d}\sum_{j=1}^d|X_{1,j}-X_{2,j}|^p  \bigg)^{1/p}-\bigg( \frac{1}{d}\sum_{j=1}^dE_{\bX_1}|X_{1,j}-X_{2,j}|^p  \bigg)^{1/p}  \Bigg)^2\\
&\hspace{1cm}\leq 4\bigg(  \frac{1}{d^{1/p}}\sum_{j=1}^d E_{\bX_1}|X_{1,j}-X_{2,j}|\bigg)^2+ 4\bigg(\frac{1}{d^{1/p}}\sum_{j=1}^d(E_{\bX_1}|X_{1,j}-X_{2,j}|^p)^{1/p}   \bigg)^2.\notag
\end{align}
The upper bound in \eqref{tay2}  and the Cauchy--Schwarz inequality yield
\begin{align*}
E&\Bigg[\bigg( \frac{1}{d^{1/p}}\sum_{j=1}^d E_{\bX_1}|X_{1,j}-X_{2,j}|\bigg)^2I\{A\}\Bigg]\\
&\leq \frac{1}{d^{2/p}}
\Bigg(E\bigg(\sum_{j=1}^d E_{\bX_1}|X_{1,j}-X_{2,j}|\bigg)^4\Bigg)^{1/2}(P\{A\})^{1/2}\\
&\leq \frac{1}{d^{2/p}}\left(c_{13}d^{-\alpha/2}\right)^{1/2}\Bigg(E\bigg(\sum_{j=1}^d E_{\bX_1}|X_{1,j}-X_{2,j}|\bigg)^4\Bigg)^{1/2}.
\end{align*}
Applying Cauchy's inequality (cf.\ Abadir and  Magnus, 2005, p.\ 324),  we get
\begin{align*}
\bigg(\sum_{j=1}^d E_{\bX_1}|X_{1,j}-X_{2,j}|\bigg)^4&\leq \bigg(d \sum_{j=1}^d (E_{\bX_1}|X_{1,j}-X_{2,j}|)^2 \bigg)^2\\
&\leq d^3\sum_{j=1}^d(E_{\bX_1}|X_{1,j}-X_{2,j}|)^4\\
&\leq 2^4 d^3\sum_{j=1}^d\left(|X_{2,j}|^4+(E|X_{1,j}| )^4  \right)\\
& = O(d^4).
\end{align*}
Hence  for the first term on the right-hand side of \eqref{tay3}, on the set $A$, we obtain
\begin{align*}
E&\bigg[\bigg( \frac{1}{d^{1/p}}\sum_{j=1}^d E_{\bX_1}|X_{1,j}-X_{2,j}|\bigg)^2I\{A\}\bigg]\leq c_{14}d^{2-2/p-\alpha/4} = o(d^{-1})
\end{align*}
For the second term on the right-hand side of \eqref{tay3} we have by Jensen's inequality
\begin{align*}
E\bigg(\frac{1}{d^{1/p}}\sum_{j=1}^d(E_{\bX_1}|X_{1,j}-X_{2,j}|^p)^{1/p}   \bigg)^4
&=d^{4-4/p}E\bigg[\frac{1}{d}\sum_{j=1}^d\left(E_{\bX_1}|X_{1,j}-X_{2,j}|^p\right)^{1/p}   \bigg]^4\\
&\leq d^{4-4/p}\frac{1}{d}\sum_{j=1}^dE\left(E_{\bX_1}|X_{1,j}-X_{2,j}|^p\right)^{4/p}\\
&\leq 2^p d^{4-4/p}\frac{1}{d}\sum_{j=1}^d\left( E|X_{1,j}|^4+\left(E|X_{2,j}|^p\right)^{4/p}   \right)\\
&=O\left(d^{4-4/p}\right).
\end{align*}
Thus we get
\begin{align*}
E&\bigg[\bigg(\frac{1}{d^{1/p}}\sum_{j=1}^d(E_{\bX_1}|X_{1,j}-X_{2,j}|^p)^{1/p}   \bigg)^2I\{A\}\bigg]\\
&\hspace{1cm}\leq \bigg(E\bigg(\frac{1}{d^{1/p}}\sum_{j=1}^d(E_{\bX_1}|X_{1,j}-X_{2,j}|^p)^{1/p}   \bigg)^4\bigg)^{1/2}\left(P\{A\}\right)^{1/2}\\
&\hspace{1cm}=O\left( d^{2-2/p-\alpha/4}  \right)\\
& \hspace{1cm}= o(d^{-1})
\end{align*}
We now turn to estimates of \eqref{tay0} on the set $\bar A$. By the definition of $A$ in \eqref{tay2}, and using that $\alpha>4$, we obtain
\begin{align*}
E&\Bigg(E_{\bX_1}\Bigg[\bigg(\frac{1}{d}\sum_{j=1}^d|X_{1,j}-X_{2,j}|^p\bigg)^{1/p}-\bigg(\frac{1}{d}\sum_{j=1}^dg_j(X_{1,j})    \bigg)^{1/p}I\{\bar{A}\}\Bigg]\Bigg)^2\\
&\leq
c_{15}E\Bigg( E_{\bX_1}\bigg(\frac{1}{d}\sum_{j=1}^d[|X_{1,j}-X_{2,j}|^p-g_j(X_{1,j})]\bigg)^2 \Bigg)^2.\\
&\leq c_{15}E\Bigg( E_{\bX_1}\bigg|\frac{1}{d}\sum_{j=1}^d[|X_{1,j}-X_{2,j}|^p-g_j(X_{1,j})]\bigg|^\alpha \Bigg)^{4/\alpha}\\
&\leq c_{15}\Bigg( E\bigg|\frac{1}{d}\sum_{j=1}^d[|X_{1,j}-X_{2,j}|^p-g_j(X_{1,j})]\bigg|^\alpha \Bigg)^{4/\alpha}\\
& \leq c_{16}d^{-2}.
\end{align*}
where we used Assumption \ref{as4} on the last line.
This completes the proof of
\begin{align*}
\E \Bigg( E_{\bX_1}\bigg( \frac{1}{d}\sum_{j=1}^d|X_{1,j}-X_{2,j}|^p  \bigg)^{1/p}-\bigg( \frac{1}{d}\sum_{j=1}^dE_{\bX_1}|X_{1,j}-X_{2,j}|^p  \bigg)^{1/p}  \Bigg)^2=o\left(\frac{1}{d}\right),
\end{align*}
as $d\to \infty$.  For the second term in \eqref{e:zeta_3term_expansion}, similar arguments give
\begin{align*}
\Bigg(\bigg( \frac{1}{d}\sum_{j=1}^dE|X_{1,j}-X_{2,j}|^p  \bigg)^{1/p}-E\bigg( \frac{1}{d}\sum_{j=1}^d|X_{1,j}-X_{2,j}|^p  \bigg)^{1/p}\Bigg)^2
=o\left(\frac{1}{d}\right),
\end{align*}
as $d\to \infty$. Now, for the third term in \eqref{e:zeta_3term_expansion}, it follows from the definition of $g_j$ and from a Taylor expansion
\begin{align*}
&\hspace{-.5cm}\bigg( \frac{1}{d}\sum_{j=1}^dE_{\bX_1}|X_{1,j}-X_{2,j}|^p  \bigg)^{1/p}
-\bigg( \frac{1}{d}\sum_{j=1}^dE|X_{1,j}-X_{2,j}|^p  \bigg)^{1/p}\\
&=\bigg( \frac{1}{d}\sum_{j=1}^d g_j(X_{1,j}) \bigg)^{1/p}
-\bigg( \frac{1}{d}\sum_{j=1}^d Eg_j(X_{1,j})\bigg)^{1/p}\\
&=\frac{1}{p}\bigg(\frac{1}{d}\sum_{j=1}^d Eg_j(X_{1,j}) \bigg)^{-1+1/p}\frac{1}{d}\sum_{j=1}^d [g_j(X_{1,j})-Eg_j(X_{1,j})]\\
&\hspace{1cm}+\frac{1}{p}\left(1-\frac{1}{p}\right)\xi^{-2+1/p}\bigg(\frac{1}{d}\sum_{j=1}^d [g_j(X_{1,j})-Eg_j(X_{1,j})]\bigg)^{{2}},
\end{align*}
where $\xi$ satisfies
\begin{align*}
\bigg|\xi-\frac{1}{d}\sum_{j=1}^d Eg_j(X_{1,j})\bigg|\leq \bigg|\frac{1}{d}\sum_{j=1}^d [g_j(X_{1,j})-Eg_j(X_{1,j})]\bigg|.
\end{align*}
Repeating our previous arguments we get
\begin{align*}
E&\Bigg(\bigg( \frac{1}{d}\sum_{j=1}^dE_{\bX_1}|X_{1,j}-X_{2,j}|^p  \bigg)^{1/p}
-\bigg( \frac{1}{d}\sum_{j=1}^dE|X_{1,j}-X_{2,j}|^p  \bigg)^{1/p}\Bigg)^2\\
&=
\frac{1}{p^2}\bigg(\frac{1}{d}\sum_{j=1}^d Eg_j(X_{1,j}) \bigg)^{-2+2/p}E\bigg(\frac{1}{d}\sum_{j=1}^d [g_j(X_{1,j})-Eg_j(X_{1,j})]\bigg)^2
+o\left(\frac{1}{d}\right),
\end{align*}
completing the proof of \eqref{ori2}.  \eqref{ori3} follows from \eqref{ori}.
\end{proof}

We use the following lemma in the proofs in Sections \ref{sec-pr1} and \ref{sec-pr2}; it is analogous to Lemma \ref{appe-mom} for the case of strongly dependent coordinates.
\begin{lemma} \label{appesecmom} If Assumptions \ref{as6} and \ref{as7} hold, then, as $d\to \infty$
\begin{equation}\label{e:strongmoment_1}
\E\psi_{1,2}^4\leq c,
\end{equation}
\begin{equation}\label{e:strongmoment_2}
E\zeta_1^2=\gamma^2(1+o(1))
\end{equation}
and
\begin{equation}\label{e:strongmoment_3}
E|\zeta_1|^4\leq c
\end{equation}
with some $c$.
\end{lemma}
\begin{proof} Statements \eqref{e:strongmoment_1} and \eqref{e:strongmoment_3} are straightforward consequences of Assumption \ref{as6}.  To establish statement \eqref{e:strongmoment_2}, for a given function $Y=\{Y(t),0\leq t \leq 1\}$, we write $\|Y\|_p=\big(\int_0^1 |Y(t)|^pdt\big)^{1/p}$. Let $Y_{\ell,d}(t) = \sum_{i=1}^d \mathbf 1_{(t_{i-1},t_i]}(t) X_{\ell,i} =\sum_{i=1}^d \mathbf 1_{(t_{i-1},t_i]}(t) Y_\ell(t_i)$, and recall $\zeta_1 = H(\mathbf X_1)-\theta$. Reexpress
\begin{align}
H(\bX_1) &=  [H(\bX_1) - \E_{\bX_1}\| Y_{1,d}-Y_{2,d}\|_{p}] +[\E_{\bX_1}\| Y_{1,d}-Y_{2,d}\|_{p}- \mathcal H(Y_{1,d})]\notag\\
&\quad + [ \mathcal H(Y_{1,d})-  \mathcal H(Y_{1})] +  \mathcal H(Y_{1})\label{e:H(X_1)_reexpressed}
\end{align}
With $Z_j=X_{1,j}-X_{2,j}$, for the first term in \eqref{e:H(X_1)_reexpressed}, observe $$\E \big| H(\bX_1)-  \E_{\bX_1}\| Y_{1,d}-Y_{2,d}\|_{p}\big|^2 = \E \bigg| \Big(\sum_{\ell=1}^d|Z_\ell|^p d^{-1} \Big)^{1/p} -\Big(\sum_{\ell=1}^d|Z_\ell|^p (t_\ell-t_{\ell-1})\Big)^{1/p} \bigg|^{ 2}
$$
$$\leq \E \bigg( \sum_{\ell=1}^d|Z_\ell|^p\bigg|\frac{1}{d^{1/p}}-\frac{1}{(t_\ell-t_{\ell-1})^{1/p}}\bigg|^{p}\bigg)^{2/p} 
$$
$$\leq \max_{1\leq i \leq d}\bigg|1-\frac{d^{1/p}}{(t_i-t_{i-1})^{1/p}}\bigg|^{2}\E\Big(\frac{1}{d}\sum_{\ell=1}^d|Z_{\ell}|^p\Big)^{2/p}  \to 0,$$
since $\max_{1\leq i \leq d}\big|1-d^{1/p}/{(t_i-t_{i-1})^{1/p}}\big|\to 0$ by Assumption \ref{as6}(ii), and $$\E\Big( d^{-1}\sum_{\ell=1}^d|Z_{\ell}|^p\Big)^{2/p}\leq C \E\Big(d^{-1}\sum_{\ell=1}^d|X_{1,\ell}|^p\Big)^{2/p}=O(1)$$ due to Assumption \ref{as6}(iv). 
 For the second term in \eqref{e:H(X_1)_reexpressed}, note
$$\big|\E_{\bX_1}\| Y_{1,d}-Y_{2,d}\|_{p}- \mathcal H(Y_{1,d})\big| =\big|\E_{\mathbf X_1}\big( \| Y_{1,d}-Y_{2,d}\|_{p}- \| Y_{1,d}-Y_{2}\|_p\big)\big|
\leq \E  \| Y_{2,d}-Y_{2}\|_p .$$
By Assumption \ref{as6}(iv),  $\E \| Y_{2,d}\|_p \to \E \|Y_{2}\|_p$, and also $ \| Y_{2,d}-Y_{2}\|_p\to 0$ a.s., which together imply $ \E  \| Y_{2,d}-Y_{2}\|_p  \to 0$.  Thus, $E[H(\bX_1) - \mathcal H(Y_{1,d})]^2\to 0$.  For the third term in \eqref{e:H(X_1)_reexpressed}, a similar argument shows $E [ \mathcal H(Y_{1,d})-  \mathcal H(Y_{1})]^2 \to 0$, implying $\E[H(\bX_1) -\mathcal H(Y_{1})]^2\to0$, which gives \eqref{e:strongmoment_2}.
\end{proof}

\medskip

\bibliographystyle{agsm}
\bibliography{Upaper}
\end{document}